\documentclass[12pt,a4paper]{article}
  \usepackage{amssymb,amsmath,mathrsfs,amsthm}
  \usepackage{authblk}

   \theoremstyle{plain}
   \newtheorem{thm}{Theorem}[section]
   \newtheorem{lem}[thm]{Lemma}

   \theoremstyle{definition}

   \theoremstyle{remark}
   \newtheorem{rmk}{Remark}[section]

  \newcommand{\ve}{\varepsilon}
  \newcommand{\bfa}{\mathbf{a}}
  \newcommand{\bfb}{\mathbf{b}}
  \newcommand{\bff}{\mathbf{f}}
  \newcommand{\bfg}{\mathbf{g}}
  \newcommand{\bfp}{\mathbf{p}}
  \newcommand{\bfq}{\mathbf{q}}
  \newcommand{\bfQ}{\mathbf{Q}}
  \newcommand{\bfu}{\mathbf{u}}
  \newcommand{\bfv}{\mathbf{v}}
  \newcommand{\bfw}{\mathbf{w}}
  \newcommand{\bfx}{\mathbf{x}}
  \newcommand{\bfnu}{\boldsymbol{\nu}}

\begin{document}

\title{Reconstruction of penetrable inclusions in elastic waves by boundary
measurements}
\author{Rulin Kuan\footnote{D97221002@ntu.edu.tw}}
\affil{Department of Mathematics, National Taiwan University, Taipei 106, Taiwan}
\date{}
\maketitle

\begin{abstract}
We use Ikehata's enclosure method to reconstruct penetrable unknown inclusions in a plane elastic body in time-harmonic waves. Complex geometrical optics solutions with complex polynomial phases are adopted as the probing utility.
In a situation similar to ours, due to the presence of a zeroth order term in the equation, some technical assumptions need to be assumed in early researches. 
In a recent work of Sini and Yoshida, they succeeded in abandoning these assumptions by using a different idea to obtain a crucial estimate. In particular the boundaries of the inclusions need only to be Lipschitz. In this work we apply the same idea to our model. It's interesting that, with more careful treatment, we find the boundaries of the inclusions can in fact be assumed to be only continuous.\\

\end{abstract}
\setcounter{equation}{0}
\section{Introduction}
In this paper we consider the inverse problem of reconstructing penetrable unknown inclusions in a plane elastic body by boundary measurements. In \cite{UW2007} and \cite{UWW2009}, the same problem is considered in the context of elastostatics. In the present work we shall consider the situation when time-harmonic waves are applied. The mathematical model is described in the following.

\subsection{Mathematical model}
Let $\Omega\subset\mathbb{R}^2$ be a bounded domain (open connected set) 
occupied by our object, which consists of an elastic body as background and some unknown inclusions therein. For simplicity we assume $\Omega$ has $\mathcal{C}^{\infty}$ boundary. The background elastic body will be assumed to be homogeneous and isotropic with Lam\'{e} constants denoted by $\lambda_0$ and $\mu_0$. Denote the region of unknown inclusions by $D$. $D$ is an open subset of $\Omega$ with $\bar{D}\subset\Omega$. 
The inclusions are also assumed to be isotropic but may be inhomogeneous. Denote the differences between the Lam\'{e} coefficients of the inclusions and the background by $\lambda_D$ and $\mu_D$, which are assumed to be in $L^{\infty}(\Omega)$, with $\lambda_D=\mu_D=0$ on $\Omega\setminus\bar{D}$. So
the Lam\'{e} coefficients $\lambda$ and $\mu$ of the whole object on $\Omega$ are given by
\[
\lambda=\lambda_0+\lambda_D
\quad\textup{and}\quad
\mu=\mu_0+\mu_D.
\]

For simplicity we also assume our object has unit density. Now, consider we send a time-harmonic elastic wave with time dependence $e^{ikt}$ into $\Omega$. By singling out the space part we have the displacement field $\bfu$, which is a two-component vector-valued function, satisfying
\begin{align}
\nabla\cdot(\sigma(\bfu))+k^2\bfu=0\quad\mbox{in }\Omega.
\label{tildeE}\end{align} 
Here, for any displacement field $\bfv$ (which we will assumed to be a column vector), $\sigma(\bfv)$ is the corresponding stress tensor, which is represented by a $2 \times 2$ matrix:
\[
\sigma(\bfv)=\lambda(\nabla\cdot\bfv)I_2 + 2\mu\epsilon(\bfv),
\]
where $I_2$ is the $2 \times 2$ identity matrix and $\epsilon(\bfv)=\frac{1}{2}(\nabla\bfv+(\nabla\bfv)^T)$ denotes the infinitesimal strain tensor. Note that for $\bfv=(v_1,v_2)^T$, $\nabla\bfv$ denotes the $2\times 2$ matrix whose $j$-th row is $\nabla v_j$ for $j=1,2$. And for a $2\times 2$ matrix function $A$, $\nabla \cdot A$ denotes the column vector whose $j$-th component is the divergence of the $j$-th row of $A$ for $j=1,2$. 

For $D=\emptyset$, that is for the case with no inclusion, the corresponding displacement field will usually be denoted by $\bfu_0$, which satisfies
\begin{align}
\nabla\cdot(\sigma_0(\bfu_0))+k^2\bfu_0=0\quad\mbox{in }\Omega,
\label{E0}\end{align}
where 
\[
\sigma_0(\bfv)=\lambda_0(\nabla\cdot\bfv)I_2+2\mu_0\epsilon(\bfv)
\]
for any displacement field $\bfv$. Accordingly, we will use $\sigma_D(\bfv)$ to denote $\sigma(\bfv)-\sigma_0(\bfv)$, i.e.
\[
\sigma_D(\bfv) = \lambda_D(\nabla\cdot\bfv)I_2 + 2\mu_D\ve(\bfv).
\]

We assume $\lambda_0,\mu_0$ and $\lambda,\mu$ satisfy the conditions
\begin{equation}\label{elliptic}
\begin{aligned}
&\lambda_0+2\mu_0>0,\mbox{ }\mu_0>0,\mbox{ and}\\
&\lambda+2\mu>0,\mbox{ }\mu>0\quad\mbox{on }\Omega,
\end{aligned}
\end{equation}
which ensure respectively that $-\nabla\cdot\sigma_0$ and $-\nabla\cdot\sigma$ are strongly elliptic operators. In particular the two operators both have at most countably many Dirichlet eigenvalues. As a consequence, we can readily choose (and will choose) $k\in\mathbb{R}$ so that $k^2$ is neither an eigenvalue of $-\nabla\cdot\sigma_0$ nor an eigenvalue of $-\nabla\cdot\sigma$. In this situation, the Dirichlet boundary value problems corresponding to (\ref{tildeE}) and (\ref{E0}) have unique solutions (see e.g. Ch.4 of \cite{M2000book}). Thus we can define the Dirichlet-to-Neumann maps $\Lambda_D$ and $\Lambda_{\emptyset}$, both from $H^{\frac{1}{2}}(\partial\Omega)^2$ to $H^{-\frac{1}{2}}(\partial\Omega)^2$, by
\begin{equation}\label{dnmap}
\Lambda_D \bff=\sigma(\bfu)\bfnu|_{\partial\Omega}
\quad\textup{and}\quad
\Lambda_{\emptyset}\bff=\sigma_0(\bfu_0)
\bfnu|_{\partial\Omega},
\end{equation}
where $\bfnu$ is the unit outer normal on $\partial\Omega$ and $\bfu$ and $\bfu_0$ solve respectively (\ref{tildeE}) and (\ref{E0}) with Dirichlet boundary data $\bff$. The goal is to determine the unknown inclusions from the knowledge of $\Lambda_D$ and $\Lambda_{\emptyset}$.

\subsection{The method and improvement}
We will utilize the enclosure-type method to reconstruct the unknown inclusions. Such kind of methods are initiated by Ikehata and have been successfully applied to a various type of reconstruction problems, see for example \cite{I1998, I1999, I2005-2, II2008}. In this method, complex geometrical optics (CGO) solutions usually play the important role of the probing utility. In his early works, Ikehata use the Calder\'{o}n type harmonic function \cite{C1980} $e^{x\cdot(\omega+i\omega^{\perp})}$ with $\omega\in \mathbb{S}^{n-1}$ \cite{I1999,I1999-2}. It looks like one uses lines (planes) to enclose the obstacle (and hence the name). As a consequence a connected inclusion is required to be convex for a complete identification, and in general only its convex hull can be determined. One can refer to the survey paper \cite{I2001} for detailed explanation and early development of this theory. In \cite{SW2006}, \cite{NY2007} and \cite{IINSU2007}, the authors utilize the complex spherical wave solutions and some concave parts of unknown inclusions can be determined. In \cite{UW2008}, the authors proposed a framework of constructing CGO solutions with general phases for some elliptic systems in two-dimension. In the same paper they then applied CGO solutions with complex polynomial phases to conductivity equations, and inclusions with more general shapes can be determined. This type of CGO solutions were later applied to other equations, for example \cite{UWW2009} for static elastic systems and \cite{NUW2011} for Helmholtz equations. In this work we will also apply CGO solutions with complex polynomial phases to our problem, of which the governing equations are the Helmholtz type elastic systems (\ref{tildeE}) and (\ref{E0}). 

A crucial point in our problem, as in \cite{I1999,NY2007,NUW2011}, is the presence of the zeroth order term. Due to this, some technical assumptions are needed in early researches. In particular $\partial D$ is assumed to be $\mathcal{C}^2$. However in the recent work \cite{SYpreprint} of Sini and Yoshida, by using a different idea to obtain a crucial estimate, they succeeded in abandoning these technical assumptions, and in particular $\partial D$ can be only Lipschitz. In this paper, we apply the same idea to our model. With more careful treatment, we find the boundaries of the inclusions can in fact be assumed to be only continuous. More detailed discussions are given in the remark after our main theorem, Theorem \ref{thm4.1}.

In the following we give a sketch of this paper as well as a rough idea of the whole process of the enclosure method. In section 2, we introduce a functional $E$ on $H^{\frac{1}{2}}(\partial\Omega)^2$, which will be called the indicator functional in this paper. And then we give an upper bound and a lower bound of $E$, which play central roles in the proof of the main theorem. In fact, we will construct a family $\bff_{d,h} \in H^{\frac{1}{2}}(\partial\Omega)^2$ as input data into $E$, and the limiting behavior of the output data, for various $d$, will indicate the location of $\partial D$. The construction of $\bff_{d,h}$ is based on the construction of CGO solutions for (\ref{E0}), which is given in section 3. By using the Helmholtz decomposition and the Vekua transform, this construction is much the same as in \cite{NUW2011}. The main theorem concerning the limiting behavior of $E$ on $\bff_{d,h}$, as well as discussions on the implication, the idea of proof and our improvement, are given in section 4.
\section{The indicator functional}
\setcounter{equation}{0}
In this section we introduce the functional $E$ on $H^{\frac{1}{2}}(\partial\Omega)$ defined by
\[
E(\bff)=\int_{\partial\Omega}
\left[(\Lambda_D-\Lambda_{\emptyset})\bff\right]\cdot\bar{\bff}ds,
\]
where the Dirichlet-to-Neumann maps $\Lambda_D$ and $\Lambda_{\emptyset}$ are defined in (\ref{dnmap}). $E$ will be called the indicator functional (according to Ikehata's indicator function \cite{I1999-2}), which plays a central role in the enclosure method. Intuitively, it measures, for a fixed Dirichlet boundary data, the difference between the tractions corresponding to the situations with and without $D$. 

Now let $\bfu$ and $\bfu_0 \in H^1(\Omega)^2$ satisfy (\ref{tildeE}) and (\ref{E0}) respectively with the same boundary condition 
$\bff\in H^{\frac{1}{2}}(\partial\Omega)^2$, and let $\bfw = \bfu - \bfu_0$. The goal in this section is to prove Lemma \ref{lem1.7}, which gives a lower bound and an upper bound of $E(\bff)$ in terms of $\bfu_0$ and $\bfw$. To this end, we first give two identities. Note that we use $|A|$ to denote $\left(\sum_{i,j}a_{ij}^2\right)^{1/2}$ for a matrix $A=\left(a_{ij}\right)$. 

\begin{lem}\label{lem1.6}
We have the following two identities:
\begin{align}
&\begin{aligned}\label{E.a}
&E(\bff) = \int_D\bigg{\{}
(\lambda_D+\mu_D)\left|\nabla\cdot\bfu_0\right|^2
+ 2\mu_D\left|\epsilon(\bfu_0)-\frac{1}{2}(\nabla\cdot\bfu_0)I_2\right|^2
\bigg{\}}dx \\
&\qquad\quad - \int_{\Omega}\bigg{\{}
(\lambda+\mu)\left|\nabla\cdot\bfw\right|^2
+ 2\mu\left|\epsilon(\bfw)-\frac{1}{2}(\nabla\cdot\bfw)I_2\right|^2
\bigg{\}}dx \\
&\qquad\quad + \int_{\Omega}k^2\left|\bfw\right|^2dx; 
\end{aligned} \\
&\begin{aligned}\label{E.b}
&E(\bff) = \int_D\bigg{\{}
(\lambda_D+\mu_D)\left|\nabla\cdot\bfu\right|^2 
+ 2\mu_D\left|\epsilon(\bfu)-\frac{1}{2}(\nabla\cdot\bfu)I_2\right|^2
\bigg{\}}dx \\
&\qquad\quad + \int_{\Omega}\bigg{\{}
(\lambda_0+\mu_0)\left|\nabla\cdot\bfw\right|^2
+ 2\mu_0\left|\epsilon(\bfw)-\frac{1}{2}(\nabla\cdot\bfw)I_2\right|^2
\bigg{\}}dx \\
&\qquad\quad - \int_{\Omega}k^2\left|\bfw\right|^2dx. 
\end{aligned}
\end{align}
\end{lem}
\begin{lem}\label{lem1.7}
Assume that the Lam\'{e} coefficients $\lambda_0$, $\mu_0$ and $\lambda$, $\mu$ satisfy the strong convexity condition, that is
\[
\lambda_0+\mu_0,\mbox{ }\mu_0>0 \quad\mbox{and}\quad \lambda+\mu,\mbox{ }\mu>0,
\]
then we have the following upper bound and lower bound of $E(\bff)$:
\begin{align*}
&E(\bff)\leq\int_D (\lambda_D+\mu_D)\left|\nabla\cdot\bfu_0\right|^2 dx \\
&\qquad\quad+2\int_D\mu_D
\left|\epsilon(\bfu_0)-\frac{1}{2}(\nabla\cdot\bfu_0)I_2\right|^2dx
+\int_{\Omega}k^2\left|\bfw\right|^2dx; \\
&E(\bff)\geq\int_D
\frac{(\lambda_D+\mu_D)(\lambda_0+\mu_0)}{\lambda+\mu}
\left|\nabla\cdot\bfu_0\right|^2 dx \\
&\qquad\quad+2\int_D\frac{\mu_D\mu_0}{\mu}
\left|\epsilon(\bfu_0)-\frac{1}{2}(\nabla\cdot\bfu_0)I_2\right|^2dx
-\int_{\Omega}k^2\left|\bfw\right|^2dx.
\end{align*}
\end{lem}
\begin{proof}
The upper bound of $E(\bff)$ follows immediately from (\ref{E.a}) (by omitting the second integral).
On the other hand, from (\ref{E.b}) we have
\begin{align}\label{I.1}
\begin{aligned}
E(\bff)&\geq
\int_D\bigg{\{}(\lambda_D+\mu_D)|\nabla\cdot\bfu|^2
+2\mu_D\left|\epsilon(\bfu)-\frac{1}{2}(\nabla\cdot\bfu)I_2\right|^2
\bigg{\}}dx\\
&\quad+\int_D\bigg{\{}(\lambda_0+\mu_0)|\nabla\cdot\bfw|^2
+2\mu_0\left|\epsilon(\bfw)-\frac{1}{2}(\nabla\cdot\bfw)I_2\right|^2
\bigg{\}}dx\\
&\quad-\int_{\Omega}k^2|\bfw|^2dx.
\end{aligned}
\end{align}
And the lower bound follows from the following two identities, of which the verifications are straightforward (by using $\bfw=\bfu-\bfu_0$).
\begin{itemize}
\item [(i)]
\[
\begin{aligned}
&(\lambda_D+\mu_D)\left|\nabla\cdot\bfu\right|^2
+ (\lambda_0+\mu_0)\left|\nabla\cdot\bfw\right|^2 \\
&\quad= \left( \sqrt{\lambda+\mu}\nabla\cdot\bfu
-\frac{\lambda_0+\mu_0}{\sqrt{\lambda+\mu}}
\nabla\cdot\bfu_0 \right)^2 + \frac{(\lambda_D+\mu_D)(\lambda_0+\mu_0)}{\lambda+\mu} \left|\nabla\cdot\bfu_0\right|^2.
\end{aligned}
\]
\item [(ii)]
\[
\begin{aligned}
&2\mu_D\left|\epsilon(\bfu)-\frac{1}{2}(\nabla\cdot\bfu)I_2\right|^2
+ 2\mu_0\left|\epsilon(\bfw)-\frac{1}{2}(\nabla\cdot\bfw)I_2\right|^2 \\
&\quad=\sum_{i,j}\left|\sqrt{2\mu}b_{ij} 
-\frac{2\mu_0}{\sqrt{2\mu}}b^0_{ij}\right|^2
+\frac{2\mu_D\mu_0}{\mu}\left|\epsilon(\bfu_0)
-\frac{1}{2}(\nabla\cdot\bfu_0)I_2\right|^2,
\end{aligned}
\]
where 
\[
(b_{ij}):=\epsilon(\bfu)-\frac{1}{2}(\nabla\cdot\bfu)I_2
\quad\mbox{and}\quad (b^0_{ij}):=\epsilon(\bfu_0)-\frac{1}{2}(\nabla\cdot\bfu_0)I_2.
\]
\end{itemize}
\end{proof}

For completeness we give the proof of Lemma \ref{lem1.6} in the following. 
Before doing so, note that we have the following basic formulae:
\begin{align}
\nabla\cdot(\sigma(\bfu)\bfv) &= (\nabla\cdot\sigma(\bfu))\cdot\bfv + tr(\sigma(\bfu)\nabla\bfv);\label{chain}\\
tr(\sigma(\bfu)\nabla\bfv) &= tr(\sigma(\bfv)\nabla\bfu).\label{exchange}
\end{align}
Here $tr(\cdot)$ is the trace of matrices. And
\begin{align}\label{isofor}
tr(\sigma(\bfu)\nabla\bar{\bfu})
=(\lambda+\mu)|\nabla\cdot\bfu|^2
+2\mu\left|\epsilon(\bfu)-\frac{1}{2}(\nabla\cdot\bfu)I_2\right|^2.
\end{align}
These formulae are easy to check and we shall omit the proof. Also note that we have similar formulae with $\sigma$ replaced by $\sigma_0$, $\sigma_D$, etc.

Now we give the proof of Lemma \ref{lem1.6}
\begin{proof}[Proof of Lemma \ref{lem1.6}]
First note that $\int_{\partial\Omega}\Lambda_D\bff\cdot\bar{\bff}ds$ and $\int_{\partial\Omega}\Lambda_{\emptyset}\bff\cdot\bar{\bff}ds$ are real. In fact, by definition we have
\[
\int_{\partial\Omega}\Lambda_D\bff\cdot\bar{\bff}ds 
= \int_{\partial\Omega}(\sigma(\bfu)\bfnu)\cdot\bar{\bfu}ds
= \int_{\partial\Omega}(\sigma(\bfu)^T\bar{\bfu})\cdot\bfnu dx.
\]
By divergence theorem and (\ref{chain}) we then get
\begin{equation}\label{ffbar}
\begin{aligned}
\int_{\partial\Omega}\Lambda_D\bff\cdot\bar{\bff}ds
&=\int_{\Omega}(\nabla\cdot\sigma(\bfu))\cdot\bar{\bfu} dx
 +\int_{\Omega} tr(\sigma(\bfu)\nabla\bar{\bfu})dx \\
&=\int_{\Omega}-k^2\bfu\cdot\bar{\bfu}dx
 +\int_{\Omega}tr(\sigma(\bfu)\nabla\mathbf{\bar{u}})dx,
\end{aligned}
\end{equation}
which is real. Similarly $\int_{\partial\Omega}\Lambda_{\emptyset}\bff\cdot\bar{\bff}ds$ is real.

Since $\bfu$ and $\bfu_0$ both equal $\bff$ on $\partial\Omega$, similar to (\ref{ffbar}) we have
\begin{align}
\int_{\partial\Omega}\Lambda_D\bff\cdot\bar{\bff}ds
&=\int_{\Omega}-k^2\bfu\cdot\bar{\bfu}_0 dx
+\int_{\Omega}tr(\sigma(\bfu)\nabla\bar{\bfu}_0) dx;\label{(*)'} \\
\int_{\partial\Omega}\Lambda_{\emptyset}\bff\cdot\bar{\bff}ds
&=\int_{\Omega}-k^2\bfu_0\cdot\bar{\bfu} dx
+\int_{\Omega}tr(\sigma_0(\bfu_0)\nabla\bar{\bfu}) dx.\label{123}
\end{align}
Take complex conjugation of (\ref{(*)'}) and by (\ref{exchange}) we get
\begin{align}
\int_{\partial\Omega}\Lambda_{D}\bff\cdot\bar{\bff}ds
=\int_{\Omega}-k^2\bfu_0\cdot\bar{\bfu} dx
+\int_{\Omega}tr(\sigma(\bfu_0)\nabla\bar{\bfu}) dx.
\label{(**)}\end{align}
Then subtract (\ref{123}) from (\ref{(**)}) we obtain
\begin{align}\label{E.1}
E(\bff)=\int_{\Omega}tr(\sigma_D(\bfu_0)\nabla\bar{\bfu})dx.
\end{align}
On the other hand, 
\begin{align*}
\int_{\Omega}k^2\bfw\cdot\bar{\bfw}dx 
= \int_{\Omega}(k^2\bfu - k^2\bfu_0)\cdot\bar{\bfw}dx 
= -\int_{\Omega}\nabla\cdot(\sigma(\bfu) - \sigma_0(\bfu_0))\cdot\bar{\bfw}dx.
\end{align*}
Note that $\bfw \in H^1_0(\Omega)^2$, thus integration by parts gives
\begin{align}\label{w.2}
k^2\int_{\Omega}\left|\bfw\right|^2dx
=\int_{\Omega}tr\left[(\sigma(\bfu)
-\sigma_0(\bfu_0))\nabla\bar{\bfw}\right]dx.
\end{align}
Now, substituting $\bfu=\bfw+\bfu_0$ into the right-hand side of (\ref{w.2}), and by (\ref{E.1}), we get
\begin{align}
k^2\int_{\Omega}|\bfw|^2dx
=\int_{\Omega}tr(\sigma(\bfw)\nabla\bar{\bfw})dx
-\int_{\Omega}tr(\sigma_D(\bfu_0)\nabla\bar{\bfu}_0)dx
+E(\bff).
\end{align}
And the first identity (\ref{E.a}) follows from (\ref{isofor}).

Similarly, by substituting $\bfu_0=\bfu-\bfw$ into the right-hand side of (\ref{w.2}) we will obtain (\ref{E.b}).
\end{proof}

\setcounter{equation}{0}
\section{The testing boundary data}
In this section we construct the boundary data to be input into $E$ for detecting the location of $\partial D$. For this purpose, we first introduce 
the CGO solutions with complex polynomial phases. 

\subsection{CGO solutions with complex polynomial phases}
We are to construct CGO solutions with complex polynomial phases to
\begin{align}\label{Orieq}
\nabla\cdot\sigma_0(\bfv)+k^2\bfv=\boldsymbol{0}\quad(\mbox{in }\mathbb{R}^2).
\end{align}
Suppose that $\bfv\in \mathcal{C}^{\infty}(\mathbb{R}^2)^2$ satisfies the above eqaution.
By Helmholtz decomposition, we can write
\[
\bfv=\nabla\varphi+\nabla^{\perp}\psi 
\]
for some smooth scalar functions $\varphi$ and $\psi$, where
$\nabla^{\perp}\psi:=(-\partial_2\psi,\partial_1\psi)^T$ (and here we also regard $\nabla\varphi$ as a column vector). Then $\varphi$ and $\psi$ satisfy
\[
\nabla((\lambda_0+2\mu_0)\Delta\varphi+k^2\varphi) + \nabla^{\perp}(\mu_0\Delta\psi+k^2\psi)=0.
\]
Let $k_1=\left(\frac{k^2}{\lambda_0+2\mu_0}\right)^{1/2}$ and $k_2=\left(\frac{k^2}{\mu_0}\right)^{1/2}$. From the above equation it's easy to see that conversely for any $\varphi$ and $\psi\in\mathcal{C}^{\infty}(\mathbb{R}^2)$ satisfying
\begin{align}
\left\{
\begin{array}{ll}
\vartriangle\varphi+k_1^2\varphi&=0\\
\vartriangle\psi+k_2^2\psi&=0,
\end{array}\right.
\label{(*)}\end{align}
$\bfv=\nabla\varphi+\nabla^{\perp}\psi$ is a solution to
(\ref{Orieq}).
Moreover, if $\varphi$ and $\psi$ are CGO solutions to (\ref{(*)}), then $\bfv$ is a CGO solution to (\ref{Orieq}). 

It is not difficult to construct CGO solutions to (\ref{(*)}) by using the Vekua transform, which transforms a harmonic function to a solution to a Helmholtz equation. 
Precisely, for any real constant $\omega$, the Vekua transform $T_{\omega}$ associated with $\omega$ is defined as follows: 
\[
T_{\omega}(u)(\mathbf{x})=u(\mathbf{x})-\int^1_0 u(t\mathbf{x})\frac{\partial}{\partial t}\left\{J_0(\omega |\mathbf{x}|\sqrt{1-t})\right\}dt
\]
for a function $u$, where $J_0$ is the zero order Bessel function of the first kind. 
If $u$ is a harmonic function, then $T_{\omega}(u)$ satisfies 
\[
\vartriangle (T_{\omega}(u))+\omega^2 (T_{\omega}(u))=0.
\]
This formula is derived by I. N. Vekua. 
One can refer to \cite{V1967book} for details and other related results.

In the following we adopt the same idea as in \cite{UWW2009} and \cite{NUW2011} to construct CGO solutions with complex polynomial phases. Given $N\in\mathbb{N}$ and $\beta\in\mathbb{C}$ with $|\beta|=1$,
let $\rho=\rho_{N,\beta}$ be the function on $\mathbb{R}^2$ defined by
\begin{align}\label{rhonbeta}
\rho(\mathbf{x})=\beta(x_1+ix_2)^N,
\end{align}
which, by regarding $\mathbb{R}^2$ as the complex plane $\mathbb{C}$, is a complex polynomial. 
Then we define
\begin{align}
\Gamma=\Gamma_{N,\beta}
:=\left\{r(\cos\theta,\sin\theta):r>0,|\theta-\theta_{\boldsymbol{0}}|<\frac{\pi}{2N}\right\},
\label{gamma}\end{align}
the open cone with axis $\theta=\theta_{\boldsymbol{0}}$ and open angle $\pi/N$, where $\theta_{\boldsymbol{0}}$ is such that $\beta=e^{-iN\theta_{\boldsymbol{0}}}$.
Let $\tau=\tau_{N,\beta}:=Re\{\rho_{N,\beta}\}$.
Note that in $\Gamma$ we have
\[
\tau(\mathbf{x})=r^N \cos N(\theta-\theta_{\boldsymbol{0}}) >0,
\]
where $\mathbf{x}=r(\cos\theta,\sin\theta)$.

Now for any constant $h>0$, $e^{\frac{\rho}{h}}$ is harmonic (since it is holomorphic by regarding $\mathbb{R}^2$ as $\mathbb{C}$), and hence
\[
\varphi=\varphi_{h}
:=T_{k_1}(e^{\frac{\rho}{h}})
\quad\mbox{and}\quad
\psi=\psi_{h}
:=T_{k_2}(e^{\frac{\rho}{h}})
\]
satisfy (\ref{(*)}).
Moreover, $\varphi_{h}$ and $\psi_{h}$ are CGO solutions.
In fact, we can write
\begin{align}
\varphi_{h}(\mathbf{x})=e^{\frac{\rho(\mathbf{x})}{h}}(1+R_{h,1}(\mathbf{x}))
\quad\mbox{and}\quad
\psi_{h}(\mathbf{x})=e^{\frac{\rho(\mathbf{x})}{h}}(1+R_{h,2}(\mathbf{x}))
\label{*.}\end{align}
with $R_{h,l}$ ($l=1,2$) satisfying the following estimates in $\Gamma$:
\begin{align}
\begin{split}
&\left|R_{h,l}\right|\leq h\frac{k_l^2|\mathbf{x}|^2}{4\tau(\mathbf{x})};\\
&\left|\frac{\partial R_{h,l}(\mathbf{x})}{\partial x_j}\right|
\leq\frac{Nk_{l}^2|\mathbf{x}|^{N+1}}{4\tau(\mathbf{x})}
+h\frac{k_{l}^2|x_j|}{2\tau(\mathbf{x})},\mbox{ }j=1,2.
\end{split}
\label{(R.1)}
\end{align}
These estimates are established in \cite[Lemma 2.1]{NUW2011}. In this study we will also need estimates of the second derivatives of $R_{h,l}$, which are not hard to derive in the same manner as the derivation of (\ref{(R.1)}) given in \cite{NUW2011}. 
Actually, by repeatedly applying the following well-known recurrence formulae
\[
\frac{d}{dt}(tJ_1)=tJ_0(t),\mbox{ }\frac{dJ_0(t)}{dt}=-J_1(t),\mbox{ }\forall t
\geq 0,
\]
where $J_1$ is the Bessel function of the first kind of order $1$,
and using the basic estimates
\[
|J_1(t)|\leq\frac{t}{2},\mbox{ }|J_0(t)|\leq 1,\mbox{ }\forall t\geq 0,
\]
the verification of the following estimates are direct (although somewhat lengthy):

\begin{align}
\begin{split}
\left|\frac{\partial^2 R_{h,l}(\mathbf{x})}{\partial x_i\partial x_j}\right|
&\leq\frac{1}{h}\left(\frac{k_{l}^2N^2|\mathbf{x}|^{2N}}{4\tau_N(\mathbf{x})}\right)\\
&\quad+\left(\frac{k_{l}^2N(N-1)|\mathbf{x}|^N}{4\tau_N(\mathbf{x})}
+\frac{k_{l}^2N|\mathbf{x}|^{N-1}(|x_i|+|x_j|)}{2\tau_N(\mathbf{x})}\right)\\
&\quad+h\left(\frac{k_{l}^4|x_i||x_j|}{4\tau_N(\mathbf{x})}
+\frac{k_{l}^2\delta_{ij}}{2\tau_N(\mathbf{x})}\right)
\end{split}
\label{R.4}\end{align}
in $\Gamma$, for $1\leq l,i,j\leq 2$, where $\delta_{ij}$ is the Kronecker delta.

Let $diam(\Omega)$ denote the diameter of $\Omega$. By (\ref{(R.1)}) and (\ref{R.4}) there exists a constant $C_R=C_R(\lambda_0,\mu_0,k,N,\beta,diam(\Omega))>0$ such that for any $0<h\leq 1$, $1\leq l,i,j\leq 2$ and $\mathbf{x}\in \Gamma\cap\Omega$,
\begin{align}
\begin{split}
&\left|R_{h,l}(\mathbf{x})\right|\leq h\frac{C_R}{\tau_N(\mathbf{x})};\\
&\left|\frac{\partial R_{h,l}(\mathbf{x})}{\partial x_j}\right|\leq\frac{C_R}{\tau_N(\mathbf{x})};\\
&\left|\frac{\partial^2 R_{h,l}(\mathbf{x})}{\partial x_i\partial x_j}\right|
\leq\frac{1}{h}\frac{C_R}{\tau_N(\mathbf{x})}.
\end{split}
\label{R.6}\end{align}

Now $\bfv=\bfv_h:=\nabla\varphi_h+\nabla^{\perp}\psi_h$ is a CGO solution to (\ref{Orieq}).
$\bfv_h$ can be written down explicitly as follows:
\[
\bfv_h(\mathbf{x})=e^{\frac{\rho(\mathbf{x})}{h}}
\left (
\begin{array}{l}
Q_{h,1}(\mathbf{x})\\
Q_{h,2}(\mathbf{x})
\end{array}\right ),
\]
where
\begin{align}
\begin{split}
Q_{h,1}(\mathbf{x})&=\left[\frac{1}{h}\frac{\partial\rho(\mathbf{x})}{\partial x_1}(1+R_{h,1}(\mathbf{x}))
+\frac{\partial R_{h,1}(\mathbf{x})}{\partial x_1}\right]\\
&\quad-\left[\frac{1}{h}\frac{\partial\rho(\mathbf{x})}{\partial x_2}(1+R_{h,2}(\mathbf{x}))
+\frac{\partial R_{h,2}(\mathbf{x})}{\partial x_2}\right]
\end{split}
\label{R.1h}\end{align}
and
\begin{align}
\begin{split}
Q_{h,2}(\mathbf{x})&=\left[\frac{1}{h}\frac{\partial\rho(\mathbf{x})}{\partial x_2}(1+R_{h,1}(\mathbf{x}))
+\frac{\partial R_{h,1}(\mathbf{x})}{\partial x_2}\right]\\
&\quad+\left[\frac{1}{h}\frac{\partial\rho(\mathbf{x})}{\partial x_1}(1+R_{h,2}(\mathbf{x}))
+\frac{\partial R_{h,2}(\mathbf{x})}{\partial x_1}\right].
\end{split}
\label{R.2h}\end{align}
Thus for $0<h\leq 1$ and $i=1, 2$, from (\ref{R.6}) we have the following estimates for $Q_{h,i}$ in $\Gamma\cap\Omega$: 
\begin{align}
\begin{split}
|Q_{h,i}(\mathbf{x})|&\leq\frac{2N|\mathbf{x}|^{N-1}}{h}\left(1+h\frac{C_R}{\tau(\mathbf{x})}\right)
+\frac{2C_R}{\tau(\mathbf{x})}\\
&\leq \frac{\tilde{C}_R}{h}+\frac{\tilde{C}_R}{\tau(\mathbf{x})},
\end{split}
\label{R.3}\end{align}
and for $j=1,2$
\begin{align}
\begin{split}
\left|\frac{\partial Q_{h,i}(\mathbf{x})}{\partial x_j}\right|
&\leq\frac{2N|\mathbf{x}|^{N-2}}{h}\left[(N+|\mathbf{x}|)\frac{C_R}{\tau(\mathbf{x})}+N\right]
+\frac{2C_R}{\tau(\mathbf{x})}\\
&\leq \frac{\tilde{C}_R}{h}\left(1+\frac{1}{\tau(\mathbf{x})}\right)+\frac{\tilde{C}_R}{\tau(\mathbf{x})}.
\end{split}
\label{R.7}\end{align}
where $\tilde{C}_R=\tilde{C}_R(\lambda_0,\mu_0,k,N,\beta,diam(\Omega))>0$ is a constant.

\subsection{The testing boundary data}
Note that from the discussion above the CGO solutions $\bfv_h$ are {\it controllable} in $\Gamma\cap\Omega$. In the following we go on to follow the idea in \cite{UWW2009} and \cite{NUW2011} to modify $\bfv_h$ into a family of functions localized in $\Gamma$. 

For $t>0$, let
\begin{align}\label{ellt}
\ell_t:=\{\mathbf{x}\in\Gamma:\tau(\mathbf{x})=\frac{1}{t}\},
\end{align}
the level curve of $\tau$ in $\Gamma$ at $\frac{1}{t}$. In fact, any level curve of $\tau=\tau_{N,\beta}$ has $N$ branches, and the cone $\Gamma=\Gamma_{N,\beta}$ just contains one branch with the two edges of $\Gamma$ being the asymptotes of that branch. Also note that when $t$ is larger, the curve $\ell_t$ is closer to the origin. (We refer to Figure 2 in \cite{NUW2011} for an illustration.)
Then for $d>0$ let
\begin{align}\label{gammad}
\Gamma_d=\Gamma_{N,\beta,d}:=\overline{\bigcup\limits_{0<t<d}\ell_{t}}.
\end{align}
Note that for $d_1>d_2>0$ we have $\Gamma_{d_2}\subset\Gamma_{d_1}$. 

In the following we fix an $\ve>0$ and a compact interval $J\subset (0,\infty)$.
Let $\{\phi_{d}\}_{d\in J}$ be a family of smooth cut-off functions such that
\begin{itemize}
\item[(i)] 
$0 \le \phi_{d}(\mathbf{x}) \le 1$,
\item[(ii)] 
$\phi_{d}(\mathbf{x})=1$ (resp. $0$) for $\bfx\in\Gamma_{d+\ve}$ (resp. $\bfx\in\mathbb{R}^2\setminus\Gamma_{d+2\ve}$), and
\item[(iii)]
for some $C_{\phi}>0$, we have
$|\partial_{\bfx}^{\boldsymbol{\alpha}}\phi_{d}(\bfx)| \leq C_{\phi}$
for each multiindex $\boldsymbol{\alpha}$ with $|\boldsymbol{\alpha}|\leq 2$, for each $\mathbf{x}\in\Omega$ and for each $d\in J$.
\end{itemize}
The existence of such family $\{\phi_d\}$ is obvious and we omit a precise construction.

Now let 
\begin{align}\label{pdh}
\bfp_{d,h}(\bfx)
:=\phi_{d}(\bfx)e^{-\frac{1}{hd}}\bfv_h
\in\mathcal{C}^{\infty}(\mathbb{R}^2)^2.
\end{align}
It is the traces of these $\bfp_{d,h}$ on $\partial\Omega$ that will be the testing data to be input into $E$.
In fact, we will see that the behavior of $E(\bfp_{d,h}|_{\partial\Omega})$ as $h\rightarrow 0^+$ tells whether $\Gamma_d$ intersects $D$ or not.
Now note that although $\bfp_{d,h}$ is {\it controllable} from the discussion above, it is no longer a solution to (\ref{Orieq}). However, to get information from $E(\bfp_{d,h}|_{\partial\Omega})$ we will need estimates related to the solution of (\ref{Orieq}) with boundary condition $\bfp_{d,h}|_{\partial\Omega}$. But indeed for small $h$ controllability of $\bfp_{d,h}$ gives controllability of the true solution of (\ref{Orieq}) with the same boundary condition. We explain this precisely in the following. 

Let $\bfu_{0,d,h}$ satisfy
\begin{align}\label{u0dh}
\left\{
\begin{array}{ll}
\nabla\cdot\sigma_0(\bfu_{0,d,h})
+k^2\bfu_{0,d,h}=0 & \mbox{ in }\Omega\\
\bfu_{0,d,h}=\bfp_{d,h}|_{\partial \Omega} & \mbox{ on }\partial \Omega.
\end{array}\right.
\end{align}
And let 
\begin{align}\label{wh}
\bfw_{h}=\bfp_{d,h}-\bfu_{0,d,h},
\end{align}
then $\bfw_{h}$ satisfies
\begin{align}
\left\{
\begin{array}{ll}
\nabla\cdot\sigma_0(\bfw_{h})+k^2\bfw_{h}
=\nabla\cdot\sigma_0(\bfp_{d,h})+k^2\bfp_{d,h} &\mbox{ in }\Omega\\
\bfw_{h}=0 & \mbox{ on }\partial\Omega.
\end{array}\right.
\label{wpnth}\end{align}
Let 
\begin{align}\label{gh}
\bfg_h=\nabla\cdot\sigma_0(\bfp_{d,h})+k^2\bfp_{d,h},
\end{align}
then we have the following lemma.

\begin{lem}\label{lem3.1}
There exists positive constants $C_1$ and $C$ (depending on $\Omega,\lambda_0,\mu_0,k$) such that for $0<h\leq 1$ and $d\in J$
\[
\|\bfw_h\|_{H^1(\Omega)^2}\leq C_1\|\bfg_h\|_{L^2(\Omega)^2}
\leq \frac{C}{h^2}e^{-\frac{1}{h}(\frac{1}{d}-\frac{1}{d+\ve})}.
\]
In particular, there exists $0<h_0<1$ such that for $0<h<h_0$ and $d\in J$, there is a positive constant $C'=C'(\Omega,\lambda_0,\mu_0,k)$ such that
\[
\|\bfw_h\|_{H^1(\Omega)^2}\leq C_1\|\bfg_h\|_{L^2(\Omega)^2}
\leq C'e^{-\frac{1}{h}(\frac{1}{d}-\frac{1}{d+\ve})}.
\]
\end{lem}

\begin{proof}
Note that in this paper for any two vectors $\bfa$ and $\bfb$, we define $\bfa\otimes\bfb$ to be the matrix whose $ij$-th entry is $a_i b_j$.

That $\|\bfw_{h}\|_{H^1(\Omega)^2}\leq C_1\|\bfg_h\|_{L^2(\Omega)^2}$
for some $C_1$ is classical.
So we need only to estimate $\|\bfg_h\|_{L^2(\Omega)^2}$.

Since $\bfp_{d,h}(\mathbf{x})=\phi_d(\mathbf{x})e^{-\frac{1}{hd}}\bfv_h$, we have
\[
\begin{split}
\bfg_h
&= e^{-\frac{1}{hd}}
\Big{\{}\lambda_0\nabla(\nabla\cdot(\phi_d\bfv_h))
+\mu_0\nabla\cdot(\nabla(\phi_d\bfv_h) + (\nabla(\phi_d\bfv_h))^T)
+k^2\phi_d\bfv_h\Big{\}}\\
&= e^{-\frac{1}{hd}}
\Big{\{}\lambda_0\big{[}\nabla(\nabla\phi_d\cdot\bfv_h)
+\nabla\phi_d(\nabla\cdot\bfv_h)\big{]}\\
&\qquad\qquad +\mu_0\nabla\cdot\big{[}\bfv_h\otimes\nabla\phi_d
+\nabla\phi_d\otimes\bfv_h\big{]}\\
&\qquad\qquad+\mu_0(\nabla\bfv_h+(\nabla\bfv_h)^T)
\nabla\phi_d\\
&\qquad\qquad +\phi_d[\nabla\cdot\sigma_0(\bfv_h)+k^2\bfv_h]\Big{\}}.
\end{split}
\]
Because $\nabla\cdot\sigma_0(\bfv_h)+k^2\bfv_h=0$ 
and $\nabla\phi_d=0$ outside $\Gamma_{d+2\ve}\setminus\Gamma_{d+\ve}$, we have
\begin{align}
\|\bfg_h\|_{L^2(\Omega)^2}\leq C_g e^{-\frac{1}{hd}}
\|\bfv_h\|_{H^1((\Gamma_{d+2\ve}\setminus\Gamma_{d+\ve})\cap\Omega)^2}
\label{g}\end{align}
for some positive constant $C_g=C_g(\lambda_0,\mu_0,C_{\phi})$.\\ 

By (\ref{R.3}), for $\mathbf{x}\in (\Gamma_{d+2\ve}\setminus\Gamma_{d+\ve})\cap\Omega$,
\[
\begin{split}
|\bfv_h(\mathbf{x})|&= e^{\frac{\tau(\mathbf{x})}{h}}\sqrt{|Q_{h,1}(\mathbf{x})|^2+|Q_{h,2}(\mathbf{x})|^2}\\
&\leq \sqrt{2}e^{\frac{\tau(\mathbf{x})}{h}}
\left\{\frac{\tilde{C}_R}{h}+\frac{\tilde{C}_R}{\tau(\mathbf{x})}\right\}
\leq \frac{C'_R}{h} e^{\frac{\tau(\mathbf{x})}{h}}
\end{split}
\]
for some positive constant $C'_R=C'_R(\lambda_0,\mu_0,k,diam\Omega)$.
Hence we have the following estimate:
\[
\|\bfv_h\|_{L^2((\Gamma_{d+2\ve}
\setminus\Gamma_{d+\ve})\cap\Omega)^2}
\leq\frac{C'_R}{h}\left(\int_{(\Gamma_{d+2\ve}
\setminus\Gamma_{d+\ve})\cap\Omega} e^{\frac{2\tau(\mathbf{x})}{h}}dx\right)^{\frac{1}{2}}.
\]
Similarly by (\ref{R.3}) and (\ref{R.7}) we have
\[
\|\nabla\bfv_h\|_{L^2((\Gamma_{d+2\ve}\setminus\Gamma_{d+\ve})\cap\Omega)^2}
\leq \frac{C^{''}_R}{h^2}
\left(\int_{(\Gamma_{d+2\ve}\setminus\Gamma_{d+\ve})\cap\Omega}e^{\frac{2\tau(\mathbf{x})}{h}}dx\right)^{\frac{1}{2}}.
\]
Since
\[
\int_{(\Gamma_{d+2\ve}\setminus\Gamma_{d+\ve})\cap\Omega}e^{\frac{2\tau(\mathbf{x})}{h}}dx
\leq |(\Gamma_{d+2\ve}\setminus\Gamma_{d+\ve})\cap\Omega|e^{\frac{2}{h}\frac{1}{d+\ve}},
\]
we have
\begin{align}
\|\bfg_h\|_{L^2(\Omega)^2}\leq C_ge^{-\frac{1}{hd}}\|\bfv_h\|
_{H^1((\Gamma_{d+2\ve}\setminus\Gamma_{d+\ve})\cap\Omega)^2}
\leq \frac{C}{h^2}e^{-\frac{1}{h}(\frac{1}{d}-\frac{1}{d+\ve})},
\label{vnh}\end{align}
where $C$ depends only on $\lambda_0,\mu_0,k$ and $\Omega$.

\end{proof}
\setcounter{equation}{0}
\section{The main theorem for the reconstruction of unknown inclusions}
We now come to considering our inverse problem of reconstructing $D$. For the main theorem we make the following three assumptions (in addition to those already made in the introduction) throughout this section. 
\begin{itemize}
\item[1.] We assume $\nabla\cdot\sigma_0$ and $\nabla\cdot\sigma$ satisfy the strong convexity condition (but not only the strong elliptic condition (\ref{elliptic})):
\begin{align*}
&\lambda_0+\mu_0>0,\mbox{ }\mu_0>0;\\
&\lambda+\mu>0,\mbox{ }\mu>0\quad\mbox{on }\Omega.
\end{align*}
Thus, in particular, Lemma \ref{lem1.7} applies.
\item[2.] $(\lambda_D+\mu_D)\mu_D\geq 0$ on $D$.
\item[3.] For any $\mathbf{y}\in\partial D$, there exists a ball $B_r(\mathbf{y})$ such that one of the following jump conditions holds:
\begin{align}\label{JC}
\begin{array}{llll}
\textup{(i)}&\mu_D(\mathbf{x})>r,&\lambda_D(\mathbf{x})+\mu_D(\mathbf{x})\geq 0,&\forall \mathbf{x}\in B_r(\mathbf{y})\cap D;\\
\textup{(ii)}&\mu_D(\mathbf{x})<-r,&\lambda_D(\mathbf{x})+\mu_D(\mathbf{x})\leq 0,&\forall \mathbf{x}\in B_r(\mathbf{y})\cap D.\\
\end{array}
\end{align}
\end{itemize}

Now assume the origin $\bf{0}$ is outside $\bar{\Omega}$.\footnote{In general, for $\bfa=(a_1,a_2)^T$ a point outside $\bar{\Omega}$, we should use $\rho=\beta((x_1-a_1)+i(x_2-a_2))^N$, and similar modifications of $\Gamma$, $\Gamma_d$, etc., and there is a similar result as Theorem \ref{thm4.1}. 
However as we can always set the coordinates so that $\mathbf{0}\notin\bar{\Omega}$ in practice, such consideration is not needed.} As in section 3, in the following we fix an $N\in\mathbb{N}$, a $\beta\in\mathbb{C}$ with $|\beta|=1$, an $\varepsilon>0$, and a compact interval $J\subset(0,\infty)$. And
recall the definition of $\rho$, $\Gamma$, $\ell_t$, $\Gamma_d$, and $\bfp_{d,h}$ in 
(\ref{rhonbeta}), (\ref{gamma}), (\ref{ellt}), (\ref{gammad}), and (\ref{pdh}) respectively. Also recall that we use $\tau$ to denote $Re(\rho)$. Let
\[
s_*:=\left\{
\begin{aligned}
\sup_{\mathbf{x}\in D\cap\Gamma}\tau(\mathbf{x}),&\quad\mbox{if } 
D\cap\Gamma\neq\emptyset\\
0\,,&\quad\mbox{if }D\cap\Gamma=\emptyset.
\end{aligned}\right.
\]
Note that $D\cap\Gamma\ne\emptyset$ if and only if $s_*>0$, and in this situation $\ell_{1/{s_*}}$ is a curve just touching $\partial D$, i.e. $\ell_{1/{s_*}}\cap\bar{D}=\ell_{1/s_*}\cap\partial{D}\ne\emptyset$.

For notational simplicity let $\bff_{d,h}:=\mathbf{p}_{d,h}|_{\partial \Omega}$. 
Recall that $\bfu_{0,d,h}$ satisfies
\[
\left\{
\begin{aligned}
&\nabla\cdot\sigma_0(\bfu_{0,d,h})+k^2\bfu_{0,d,h}=0\quad\mbox{in }\Omega\\
&\bfu_{0,d,h}=\bff_{d,h}\quad\mbox{on }\partial\Omega.
\end{aligned}\right.
\]
Similarly let $\bfu_{d,h}$ be the solution when the inclusion $D$ exists: 
\[
\left\{
\begin{aligned}
&\nabla\cdot\sigma(\bfu_{d,h})+k^2\bfu_{d,h}=0,\quad\mbox{in }\Omega\\
&\bfu_{d,h}=\bff_{d,h}\quad\mbox{on }\partial\Omega.
\end{aligned}\right.
\]
Now let $\bfw_{d,h}=\bfu_{d,h}-\bfu_{0,d,h}$. We have the following two inequalities from Lemma \ref{lem1.7}: 
\begin{align}
&\begin{aligned}\label{pf.U}
E(\bff_{d,h})&\leq \int_D(\lambda_D+\mu_D)|\nabla\cdot\bfu_{0,d,h}|^2dx\\
&\quad+2\int_D\mu_D
\left|\epsilon(\bfu_{0,d,h})-\frac{1}{2}(\nabla\cdot\bfu_{0,d,h})I_2\right|^2dx
+k^2\|\bfw_{d,h}\|_{L^2(\Omega)}^2;
\end{aligned}\\
&\begin{aligned}\label{pf.L}
E(\bff_{d,h})&\geq \int_D\frac{(\lambda_0+\mu_0)(\lambda_D+\mu_D)}{\lambda+\mu}
|\nabla\cdot\bfu_{0,d,h}|^2dx\\
&\quad+2\int_D\frac{\mu_0\mu_D}{\mu}
\left|\epsilon(\bfu_{0,d,h})-\frac{1}{2}(\nabla\cdot\bfu_{0,d,h})I_2\right|^2dx
-k^2\|\bfw_{d,h}\|_{L^2(\Omega)}^2.
\end{aligned}
\end{align}
They are the key to the following main theorem of this paper.

\begin{thm}\label{thm4.1}
For $d\in J$ and $h>0$ small enough, the following conclusions hold:
\begin{itemize}
\item[\textup{(A)}] If $\bar{D}\cap\Gamma_d=\emptyset$, then 
\[
|E(\bff_{d,h})|\leq C h^{-4}e^{-\frac{2}{h}(\frac{1}{d}-s_d)}
\]
for some $C>0$ independent of $h$, where $s_d=\max(\frac{1}{d+\ve},s_*)<\frac{1}{d}$.

\item[\textup{(B)}] If $D\cap\Gamma_d\ne\emptyset$ and $D$ has continuous boundary, then there exists a constant $\delta$, $0 < \delta < s_*-\frac{1}{d}$, such that
\[
|E(\bff_{d,h})| \ge Ch^{-3}e^{\frac{2}{h}(s_*-\frac{1}{d}-\delta)}
\]
for some $C>0$ independent of $h$.
\item[\textup{(B$'$)}] If $\bar{D}\cap\Gamma_d\neq\emptyset$ and $D$ has $C^{0,\alpha}$ boundary for $\frac{1}{3}<\alpha\leq 1$, then
\[
|E(\bff_{d,h})|\geq Ch^{-3+\frac{1}{\alpha}}e^{\frac{2}{h}(s_*-\frac{1}{d})}
\]
for some $C>0$ independent of $h$.
\end{itemize}
\end{thm}

Before going into the proof of Theorem \ref{thm4.1}, we give two remarks.

\begin{rmk}\label{rmk4.2}\textup{}
\begin{enumerate}
\item
From Theorem \ref{thm4.1}, we have the following conclusions.
In (A), since $s_d<\frac{1}{d}$, $|E(\bff_{d,h})|$ tends to zero as $h$ tends to zero.
On the other hand, in (B) and in (B$'$), since $s_*\geq\frac{1}{d}$, $|E(\bff_{d,h})|$ tends to infinity as $h$ tends to zero.
In particular, from (A) and (B), we have 
\begin{align}\label{td}
s_* = \inf\left\{\frac{1}{d}:\lim_{h\rightarrow 0^+}|E(\bff_{d,h})|=0\right\}.
\end{align} 
Hence, although we don't know the limiting behavior of $E(\bff_{d,h})$ when $\Gamma_d$ just touches $\partial D$, we can reconstruct $\partial D$ in principle. (Of course, due to the geometric nature of $\Gamma_d$, in fact only ``detectable'' points can be reconstructed. An explanation of this point can be found in \cite[Corollary 5.4]{UW2008}. Also see \cite{UWW2009} or \cite{NUW2011} for a reconstruction algorithm, which is easily modified to be suited for our case. We omit such discussions in this paper.) From this point of view, almost no regularity assumption on $\partial D$ is essential in the reconstruction. Nevertheless, for a complete characterization of the limiting behavior of $E(\bff_{d,h})$, we include (B$'$) in our theorem, while for this purpose more regularity assumption has to be made.
\item
We will use (\ref{pf.U}) and (\ref{pf.L}) to prove Theorem \ref{thm4.1}. Roughly speaking we have better knowledge of ${\bfu}_{0,d,h}$ than $\bfw_{d,h}$, and the crucial step is to give an appropriate control of $\|\bfw_{d,h}\|_{L^2(\Omega)}$ in terms of ${\bfu}_{0,d,h}$. For this purpose, in the corresponding parts of early researches, e.g. \cite{I1999,NY2007,NUW2011}, some technical assumptions (precisely, positivity of the relative curvature and finiteness of the number of touching points of $\ell_{1/s_*}$ (or say $\Gamma_{1/s_*}$) and 
$\partial D$) have to be made. In particular $\partial D$ is usually assumed to be $\mathcal{C}^2$. (In order to apply CGO solutions with complex polynomial phases, even more technicalities are involved. For example, in 
\cite[Lemma 3.7]{NUW2011}, the authors proposed an estimate which is based on a rather technical result in \cite{LN2003}.) In \cite{SYpreprint}, Sini and Yoshida came up with a totally different method to control $\|\bfw_{d,h}\|_{L^2(\Omega)}$ (while they did not adopt CGO solutions with complex polynomial phases). Precisely, they proposed (in our terminology)
\begin{equation}\label{syest}
\|\bfw_{d,h}\|_{L^2(\Omega)} \le C \|{\bfu}_{0,d,h}\|_{W^{1,p}(D)}
\end{equation}
for some $p<2$, which was proved by using an $L^p$ regularity estimate of Meyers and the Friedrichs' inequality. In this way the technical assumptions on the touching point are no more needed and $\partial D$ can be assumed to be only Lipschitz. Inspired by this result, we tried to adopt their idea in our situation. We find it's interesting that, with more careful treatment, we find the boundaries of the inclusions can in fact be assumed to be only continuous. Moreover, we find in the case of $\Gamma_d$ just touching $\partial D$, the regularity assumption on $\partial D$ can be reduced to be $C^{0,\alpha}$ for any $\alpha\in(\frac{1}{3},1]$. 

\end{enumerate}
\end{rmk}

To save notation, in the remaining of this paper we will freely use $C$ to denote a constant, which may represent different values at different places. 

The following lemma is just (\ref{syest}), we give the proof here for the sake of completeness.
\begin{lem}\label{(w)}
There exist constants $C>0$ and $1 \le q_0 < 2$ such that for $q_0 < q \le 2$,
\[
\|\bfw\|_{L^2(\Omega)} \le C \|\nabla\bfu_0\|_{L^{q}(D)},
\] 
whenever $\bfu$ and $\bfu_0 \in H^1(\Omega)^2$ satisfy \textup{(\ref{tildeE})} and \textup{(\ref{E0})} respectively, $\bfu$ and $\bfu_0$ have the same traces on $\partial\Omega$, and $\bfw = \bfu-\bfu_0$.
\end{lem}
\begin{proof}
Let $\bfq$ be the element in $H^1_0(\Omega)^2$ satisfying
\begin{align*}
\nabla\cdot(\sigma(\bfq))+k^2\bfq=\bar{\bfw} \quad\mbox{in }\Omega.
\end{align*}
Then, by taking inner product with $\bfw$ and integration by parts, we have
\begin{align}
\int_{\Omega}|\bfw|^2dx
&=-\int_{\Omega}tr(\sigma(\bfq)\nabla\bfw)dx
+k^2\int_{\Omega}\bfq\cdot\bfw dx \notag \\
&=-\int_{\Omega}tr(\sigma(\bfw)\nabla\bfq)dx
+k^2\int_{\Omega}\bfq\cdot\bfw dx. \label{ipwithw}
\end{align}
On the other hand, note that
\[
\nabla\cdot\sigma(\bfw)+k^2\bfw = -\nabla\cdot\sigma_D(\bfu_0),
\]
which, by taking inner product with $\bfq$ and integration by parts, gives
\begin{align}\label{ipwithq}
-\int_{\Omega} tr(\sigma(\bfw)\nabla\bfq) dx 
+ k^2\int_{\Omega} \bfw\cdot\bfq dx
= \int_{\Omega} tr(\sigma_D(\bfu_0)\nabla\bfq) dx.
\end{align}
From (\ref{ipwithw}) and (\ref{ipwithq}) we get
\[
\int_{\Omega}|\bfw|^2dx = \int_{\Omega} tr(\sigma_D(\bfu_0)\nabla\bfq) dx.
\]
Then by H\"{o}lder's inequality we have for any $1 \le p \le \infty$
\begin{align}\label{w2}
\int_{\Omega} |\bfw|^2 dx \le \|\sigma_D(\bfu_0)\|_{L^{q}(D)}\|\nabla\bfq\|_{L^p(\Omega)},
\end{align}
where $q$ is the conjugate exponent of $p$.

Now let $\bfQ=\bfq$. By definition of $\bfq$ we have
\[
\left\{
\begin{aligned}
&\nabla\cdot(\sigma(\bfQ))=\bar{\bfw}-k^2\bfq\quad\mbox{in }\Omega\\
&\bfQ=0\quad\mbox{on }\Omega.
\end{aligned}\right.
\]
Then by \cite[Theorem 1]{M1963}, there exist $p_0>2$ such that for each $2 \le p < p_0$,
\begin{align}\label{lpest}
\|\nabla\bfq\|_{L^p(\Omega)}=\|\nabla\bfQ\|_{L^p(\Omega)}
\le C\left\{\|\bfq\|_{L^2(\Omega)}+\|\bfw\|_{L^2(\Omega)}\right\}
\end{align}
for some $C = C(k,\lambda,\mu) > 0$. Note that also by definition of $\bfq$, we have 
$\|\bfq\|_{L^2(\Omega)} \le C\|\bfw\|_{L^2(\Omega)}$ for some 
$C = C(k,\lambda,\mu) > 0$ (see e.g. \cite[Section 6.2, Theorem 6]{E1998book}). So from (\ref{lpest}) we have
\begin{align}\label{LpQ}
\|\nabla\bfq\|_{L^p(\Omega)} \le C\|\bfw\|_{L^2(\Omega)}
\end{align}
for some $C = C(k,\lambda,\mu) > 0$. Combining (\ref{w2}) and (\ref{LpQ}), we have
\[
\|\bfw\|_{L^2(\Omega)}^2 \le C\|\nabla\bfu_0\|_{L^{q}(D)^2}\|\bfw\|_{L^2(\Omega)}
\]
for some $C = C(k,\lambda,\mu) > 0$ and $2 \le p < p_0$, and therefore
\[
\|\bfw\|_{L^2(\Omega)} \le C \|\nabla\bfu_0\|_{L^{q}(D)}
\]
for some $C = C(k,\lambda,\mu) > 0$ and $q_0 < q \le 2$, where $1 \le q_0 < 2$ is the conjugate exponent of $p_0$.
\end{proof}

\begin{rmk}
Remember that we assume $\Omega$ has a smooth boundary for simplicity. In fact, 
it is so assumed only to allow direct application of the $L^p$ estimates in \cite{M1963}. In other words, the regularity condition on $\partial\Omega$ really required is just that guarantees the validity of (\ref{lpest}).
\end{rmk}

To make the proof of Theorem \ref{thm4.1} more concise, some computational results are collected in the following lemma. 
\begin{lem}\label{est-data}For $d \in J$, we have the following conclusions.
\begin{itemize}
\item[\textup{(i)}]
There exsists a constant $C>0$ such that for $q>0$ and 
$0 < h \le 1$, we have
\begin{align}\label{gradp}
\|\nabla\bfp_{d,h}\|_{L^q(D)}
\leq Ch^{-2}\left(
\int_{D\cap\Gamma_{d+2\ve}}e^{\frac{q}{h}(\tau(\bfx)-\frac{1}{d})}dx
\right)^\frac{1}{q}.
\end{align}
In particular, since $s_*\ge\tau(\bfx)$ for $\bfx\in{D\cap\Gamma_{d+2\ve}}$, we have
\[
\left\|\nabla\bfp_{d,h}\right\|_{L^q(D)}
\leq Ch^{-2}e^{\frac{1}{h}(s_*-\frac{1}{d})}
\]
for some $C>0$ independent of $h$.
\item[\textup{(ii)}]
There exist positive constants $c$ and $C$ such that, for 
$0 < h \le 1$ and for any open set $U$ with 
$D\cap\Gamma_{d+\ve}\cap U\ne\emptyset$, we have
\begin{align}\label{cC}
\begin{aligned}
&\left\|
\epsilon(\bfp_{d,h})-\frac{1}{2}(\nabla\cdot\bfp_{d,h})I_2
\right\|^2_{L^2(D\cap\Gamma_{d+\ve}\cap U)}\\
&\qquad\quad\geq (ch^{-4}-Ch^{-2})
\int_{D \cap \Gamma_{d+\ve} \cap U} e^{\frac{2}{h}(\tau(\mathbf{x})-\frac{1}{d})}dx.
\end{aligned}
\end{align}
\item[\textup{(iii)}] 
There exist constants $C>0$ and $q_0<2$ such that for each 
$q_0 < q \le 2$ and $0 < h \ll 1$, we have
\begin{align}\label{wn}
\|\bfw_{d,h}\|^2_{L^2(\Omega)}
\le Ce^{\frac{2}{h}(\frac{1}{d+\ve}-\frac{1}{d})}
+Ch^{-4}\left(\int_{D\cap\Gamma_{d+2\ve}}e^{\frac{q}{h}(\tau(\bfx)-\frac{1}{d})}dx\right)^{2/q}.
\end{align}
\end{itemize}
\end{lem}

\begin{proof}
\begin{itemize}
\item[(i)]
Remember that 
\[
\bfp_{d,h}=(p^1_{d,h},p^2_{d,h})^T
=\phi_de^{\frac{1}{h}(\rho(\mathbf{x})-\frac{1}{d})}(Q_{h,1},Q_{h,2})^T,
\] 
where $Q_{h,1}$ and $Q_{h,2}$ are defined in (\ref{R.1h}) and (\ref{R.2h}) respectively. Then by definition of $\phi_d$ (in page \pageref{pdh}), for $\mathbf{x}\in D\setminus\Gamma_{d+2\ve}$, we have $\bfp_{d,h}(\bfx)=0$.
On the other hand, by (\ref{R.3}) and (\ref{R.7}), we have for $\bfx\in D\cap\Gamma_{d+2\ve}$ and $0<h\leq 1$,
\begin{align}
\begin{aligned}
|\nabla\bfp_{d,h}(\bfx)|^2
&=\sum_{j,l=1,2} \left|\frac{\partial p^j_{d,h}}{\partial x_l}\right|^2\\
&=\sum_{j,l=1,2}e^{\frac{2}{h}(\tau(\bfx)-\frac{1}{d})}
\left|\frac{\partial\phi_d(\bfx)}{\partial x_l}Q_{h,j}(\bfx)\right.\\
&\qquad\left.+\phi_d(\bfx)\left[\frac{1}{h}\frac{\partial\rho(\bfx)}{\partial x_l}Q_{h,j}(\bfx)+\frac{\partial Q_{h,j}(\bfx)}{\partial x_l}\right]\right|^2\\
&\leq Ce^{\frac{2}{h}(\tau(\bfx)-\frac{1}{d})}h^{-4},
\end{aligned}
\label{nablaV}\end{align} 
for some positive constant $C$ independent of $h$.
Since $s_*=\sup\limits_{\bfx\in D\cap\Gamma}\tau(\bfx)$ and 
$|\nabla\bfp_{d,h}|^{q}=(|\nabla\bfp_{d,h}|^2)^{q/2}$, we have for $0<h\leq 1$ 
\begin{align*}
\|\nabla\bfp_{d,h}\|_{L^q(D)}\leq Ch^{-2}
\left(\int_{D\cap\Gamma_{d+2\ve}}e^{\frac{q}{h}(\tau(\bfx)-\frac{1}{d})}dx\right)
^{\frac{1}{q}}\leq C e^{\frac{1}{h}(s_*-\frac{1}{d})}h^{-2},
\end{align*}
for some positive constant $C$ independent of $h$.

\item[(ii)] 
We can compute $\frac{\partial p^j_{d,h}}{\partial x_l}$ directly by $(\ref{R.1h})$ and (\ref{R.2h}) for $\bfx\in D\cap\Gamma_{d+2\ve}$:
\[
\begin{aligned}
\frac{\partial p^1_{d,h}(\bfx)}{\partial x_l}
&=e^{-\frac{1}{hd}}e^{\frac{\rho(\bfx)}{h}}
\left(\frac{\partial\phi_d(\bfx)}{\partial x_l}Q_{h,1}
+\frac{1}{h}\phi_d\frac{\partial\rho(\bfx)}{\partial x_l}Q_{h,1}
+\phi_d\frac{\partial Q_{h,1}(\bfx)}{\partial x_l}\right)\\
&=e^{-\frac{1}{hd}}e^{\frac{\rho(\bfx)}{h}}
\left(\frac{1}{h^2}\left(\frac{\partial\rho(\bfx)}{\partial x_1}
-\frac{\partial\rho(\bfx)}{\partial x_2}\right)
\frac{\partial\rho(\bfx)}{\partial x_l}\phi_d(\bfx)
+I_{h^{-1}}\right)
\end{aligned}
\]
and
\[
\begin{aligned}
\frac{\partial p^2_{d,h}(\bfx)}{\partial x_l}
&=e^{-\frac{1}{hd}}e^{\frac{\rho(\bfx)}{h}}
\left(\frac{\partial\phi_d(\bfx)}{\partial x_l}Q_{h,2}
+\frac{1}{h}\phi_d\frac{\partial\rho(\bfx)}{\partial x_l}Q_{h,2}
+\phi_d\frac{\partial Q_{h,2}(\bfx)}{\partial x_l}\right)\\
&=e^{-\frac{1}{hd}}e^{\frac{\rho(\bfx)}{h}}
\left(\frac{1}{h^2}\left(\frac{\partial\rho(\bfx)}{\partial x_1}
+\frac{\partial\rho(\bfx)}{\partial x_2}\right)
\frac{\partial\rho(\bfx)}{\partial x_l}\phi_d(\bfx)
+I'_{h^{-1}}\right),
\end{aligned}
\]
where
\[
\begin{aligned}
I_{h^{-1}}
&=\frac{1}{h^2}\phi_d\frac{\partial\rho}{\partial x_l}
\left(\frac{\partial\rho}{\partial x_1}R_{h,1}
-\frac{\partial\rho}{\partial x_2}R_{h,2}\right)\\
&\quad+\frac{1}{h}\frac{\partial\phi_d}{\partial x_l}
\left[\frac{\partial\rho}{\partial x_1}(1+R_{h,1})-\frac{\partial\rho}{\partial x_2}(1+R_{h,2})\right]\\
&\quad+\frac{1}{h}\phi_d
\left[\frac{\partial\rho}{\partial x_l}\left(\frac{\partial R_{h,1}}{\partial x_1}-\frac{\partial R_{h,2}}{\partial x_2}\right)
+\frac{\partial\rho}{\partial x_1}\frac{\partial R_{h,1}}{\partial x_l}-\frac{\partial\rho}{\partial x_2}\frac{\partial R_{h,2}}{\partial x_l}\right]\\
&\quad+\frac{1}{h}\phi_d\left[\frac{\partial^2\rho}{\partial x_l\partial x_1}(1+R_{h,1})-\frac{\partial^2\rho}{\partial x_l\partial x_2}(1+R_{h,2})\right]\\
&\quad+\phi_d\left(\frac{\partial^2R_{h,1}}{\partial x_l\partial x_1}-\frac{\partial^2 R_{h,2}}{\partial x_l\partial x_2}\right)
+\frac{\partial\phi_d}{\partial x_l}\left(\frac{\partial R_{h,1}}{\partial x_1}-\frac{\partial R_{h,2}}{\partial x_2}\right);
\end{aligned}
\]
\[
\begin{aligned}
I'_{h^{-1}}&=\frac{1}{h^2}\phi_d
\frac{\partial\rho}{\partial x_l}
\left(\frac{\partial\rho}{\partial x_2}R_{h,1}
+\frac{\partial\rho}{\partial x_1}R_{h,2}\right)\\
&\quad+\frac{1}{h}\frac{\partial\phi_d}{\partial x_l}
\left[\frac{\partial\rho}{\partial x_2}(1+R_{h,1})+\frac{\partial\rho}{\partial x_1}(1+R_{h,2})\right]\\
&\quad+\frac{1}{h}\phi_d
\left[\frac{\partial\rho}{\partial x_l}
\left(\frac{\partial R_{h,1}}{\partial x_2}
+\frac{\partial R_{h,2}}{\partial x_1}\right)
+\frac{\partial\rho}{\partial x_2}\frac{\partial R_{h,1}}{\partial x_l}+\frac{\partial\rho}{\partial x_1}\frac{\partial R_{h,2}}{\partial x_l}\right]\\
&\quad+\frac{1}{h}\phi_d
\left[\frac{\partial^2\rho}{\partial x_l\partial x_2}(1+R_{h,1})+\frac{\partial^2\rho}{\partial x_l\partial x_1}(1+R_{h,2})\right]\\
&\quad+\phi_d\left(\frac{\partial^2R_{h,1}}{\partial x_l\partial x_2}+\frac{\partial^2 R_{h,2}}{\partial x_l\partial x_1}\right)
+\frac{\partial\phi_d}{\partial x_l}\left(\frac{\partial R_{h,1}}{\partial x_2}+\frac{\partial R_{h,2}}{\partial x_1}\right).
\end{aligned}
\]

By (\ref{R.6}), for any $\bfx \in D \cap \Gamma_{d+2\ve}$ and $0 < h \le 1$,
\[
|I_{h^{-1}}(\bfx)|,|I'_{h^{-1}}(\bfx)|\leq Ch^{-1}
\]
for some positive constant $C$ independent of $h$.

Then we have for $\bfx \in D \cap \Gamma_{d+2\ve}$ and $0 < h \le 1$
\[
\begin{aligned}
2\bigg{|}\epsilon
&\left.(\bfp_{d,h})-\frac{1}{2}(\nabla\cdot\bfp_{d,h})I_2\right|^2\\
&\geq \left|\frac{\partial p^1_{d,h}}{\partial x_1}
-\frac{\partial p^2_{d,h}}{\partial x_2}\right|^2\\
&\geq\left|e^{\frac{1}{h}(\rho-\frac{1}{d})}\phi_dh^{-2}
\left[\frac{\partial\rho}{\partial x_1}
\left(\frac{\partial\rho}{\partial x_1}-\frac{\partial\rho}{\partial x_2}\right)
-\frac{\partial\rho}{\partial x_2}
\left(\frac{\partial\rho}{\partial x_2}+\frac{\partial\rho}{\partial x_1}\right)\right]\right|^2\\
&\quad\quad-\left|e^{\frac{1}{h}(\rho-\frac{1}{d})}\left(I_{h^{-1}}-I'_{h^{-1}}\right)\right|^2\\
&\geq e^{\frac{2}{h}(\tau-\frac{1}{d})}(c \phi_d^2 h^{-4}-2Ch^{-2})
\end{aligned}
\]
for some positive constants $c,C$ independent of $h$.
Then (ii) of this lemma is valid.

\item [(iii)]
By Lemma \ref{(w)}, there exist constants $C>0$ and $1 \le q_0 < 2$ such that 
\[
\|\bfw_{d,h}\|_{L^2(\Omega)}\leq C\|\nabla\bfu_{0,d,h}\|_{L^{q}(D)}
\]
for each $q_0 < q \le 2$.
Therefore replacing $\bfu_{0,d,h}$ by $\bfp_{d,h}-\bfw_h$ and applying H\"{o}lder's inequality, we have
\[
\begin{aligned}
\|\bfw_{d,h}\|_{L^2(\Omega)}
&\leq C\left\{\|\nabla\bfw_h\|_{L^{q}(D)}
+\|\nabla\mathbf{p}_{d,h}\|_{L^{q}(D)}\right\}\\
&\leq C \left\{\|\nabla\bfw_h\|_{L^{2}(D)}
+\|\nabla\mathbf{p}_{d,h}\|_{L^{q}(D)}\right\}\\
&\leq C\left\{\|\nabla\bfw_h\|_{H^1(D)}
+\|\nabla\mathbf{p}_{d,h}\|_{L^{q}(D)}\right\}.
\end{aligned}
\]
Then by Lemma \ref{lem3.1} and (\ref{gradp}), (\ref{wn}) follows.
\end{itemize}
\end{proof}
Now we give the proof of the main theorem.
\begin{proof}[Proof of Theorem \ref{thm4.1}]
\begin{itemize}
\item [(A)]
By (\ref{pf.L}), we have
\[
\begin{split}
-E(\bff_{d,h})&\leq \int_D\frac{-(\lambda_0+\mu_0)(\lambda_D+\mu_D)}{\lambda+2\mu}
|\nabla\cdot\bfu_{0,d,h}|^2dx\\
&\quad+2\int_D\frac{-\mu_0\mu_D}{\mu}
\left|\epsilon(\bfu_{0,d,h})-\frac{1}{2}(\nabla\cdot\bfu_{0,d,h})I_2\right|^2dx\\
&\qquad+\int_{\Omega}k^2|\bfw_{d,h}|^2dx.
\end{split}
\]
Therefore, together with (\ref{pf.U}), we have
\[
|E(\bff_{d,h})|\leq C\left\{\|\nabla\bfu_{0,d,h}\|^2_{L^2(D)}
+\|\bfw_{d,h}\|^2_{L^2(\Omega)}\right\}
\]
for some positive constant $C$ independent of $h$.
Therefore from Lemma \ref{lem3.1} and Lemma \ref{est-data}, we have, by choosing $q=2$, the following estimate:
\[
\begin{aligned}
|E(\bff_{d,h})|&\leq C\left\{\|\nabla\bfw_h\|^2_{L^2(\Omega)}
+\|\nabla\bfp_{d,h}\|^2_{L^2(D)}
+\|\bfw_{d,h}\|^2_{L^2(\Omega)}\right\}\\
&\leq C\left\{e^{\frac{2}{h}(\frac{1}{d+\ve}-\frac{1}{d})}
+h^{-4}e^{\frac{2}{h}(s_*-\frac{1}{d})}\right\}.
\end{aligned}
\]
Therefore for $0<h\leq 1$,
\[
|E(\bff_{d,h})|\leq C\left(h^{-4}e^{\frac{2}{h}(s_d-\frac{1}{d})}\right),
\]
where $s_d=\max(\frac{1}{d+\ve},s_*)$.
Moreover we notice that $\bar{D}\cap\Gamma_d=\emptyset$ implies $s_*<\frac{1}{d}$, and the conslusion (A) follows.

\item [(B)]
We first consider case (i) of (\ref{JC}) and prove the conclusion (B) from (\ref{pf.L}).

Suppose $D\cap\Gamma_d\neq\emptyset$, then $s_*>\frac{1}{d}\geq\frac{1}{d+\ve}$ since $D$ is open.
Therefore for any $\mathbf{y}\in\partial D\cap\Gamma_{d+\ve}$, each neighborhood $U_{\mathbf{y}}$ of $\mathbf{y}$ satisfies $D\cap\Gamma_{d+\ve}\cap U_{\mathbf{y}}\neq\emptyset$.
By the assumption (i) of (\ref{JC}), for each $\mathbf{y}\in \partial D$, there exists $r_{\mathbf{y}}$ such that
\begin{align}\label{ry}
\mu_D(\bfx)>r_{\mathbf{y}},\quad \lambda_D+\mu_D\geq 0,\quad \forall\bfx\in B_{r_{\mathbf{y}}}\cap D.
\end{align}

Set $K:=\partial D\cap\{\tau=s_*\}=\partial D\cap\ell_{1/s_*}$. It's easy to see that $K\neq\emptyset$.
Since $K$ is compact and is contained in 
$\cup_{y\in K}B_{r_{\mathbf{y}}}(\mathbf{y})$, there exsits $N\in\mathbb{N}$ such that $K\subset\cup_{j=1}^{N}B_{r_j}(\mathbf{y}_j)$, where $r_{\mathbf{y}_j}$ is abbreviated to $r_j$.
Let $D_R=D\setminus\cup_{j=1}^{N}B_{r_j}(\mathbf{y}_j)$, then it is easy to see that there exists $\delta '>0$ such that 
\[
\tau(\bfx)\leq s_*-\delta ' \quad\mbox{in}\quad D_R.
\] 
Therefore for $q_0<q\leq 2$ we have 
\[
\begin{split}
\int_{D\cap\Gamma_{d+2\ve}}e^{\frac{q}{h}(\tau-\frac{1}{d})}dx
&\leq\int_{D_R}e^{\frac{q}{h}(\tau-\frac{1}{d})}dx
+\sum_{j=1}^N\int_{B_{r_j}(\mathbf{y}_j)\cap D\cap\Gamma_{d+2\ve}}e^{\frac{q}{h}(\tau-\frac{1}{d})}dx\\
&\leq Ce^{\frac{q}{h}(s_*-\frac{1}{d}-\delta ')}
+N\int_{B_{r_*}(\mathbf{y}_*)\cap D\cap\Gamma_{d+2\ve}}e^{\frac{q}{h}(\tau-\frac{1}{d})}dx\\
\end{split}
\]  
for some $\mathbf{y}_*\in\{\mathbf{y}_j\}_{j=1}^N$ and $r_*\in\{r_j\}_{j=1}^N$ such that 
\[
\int_{B_{r_*}(\mathbf{y}_*)\cap D\cap\Gamma_{d+2\ve}}e^{\frac{q}{h}(\tau-\frac{1}{d})}dx
=\max_{j=1,\ldots,N}\left(\int_{B_{r_j}(\mathbf{y}_j)\cap D\cap\Gamma_{d+2\ve}}e^{\frac{q}{h}(\tau-\frac{1}{d})}dx\right).
\]
Moreover, we can compute more finely that
\[
\begin{split}
&\int_{B_{r_*}(\mathbf{y}_*)\cap D\cap\Gamma_{d+2\ve}}e^{\frac{q}{h}(\tau-\frac{1}{d})}dx\\
&=\int_{B_{r_*}(\mathbf{y}_*)\cap D\cap\Gamma_{d+\ve}}e^{\frac{q}{h}(\tau-\frac{1}{d})}dx
+\int_{B_{r_*}(\mathbf{y}_*)\cap D\cap(\Gamma_{d+2\ve}\setminus\Gamma_{d+\ve})}e^{\frac{q}{h}(\tau-\frac{1}{d})}dx\\
&\leq \int_{B_{r_*}(\mathbf{y}_*)\cap D\cap\Gamma_{d+\ve}}e^{\frac{q}{h}(\tau-\frac{1}{d})}dx  
+Ce^{\frac{q}{h}(\frac{1}{d+\ve}-\frac{1}{d})}.
\end{split}
\]
Therefore by combining the above inequalities, we have
\[
\begin{aligned}
\int_{D\cap\Gamma_{d+2\ve}}e^{\frac{q}{h}(\tau-\frac{1}{d})}dx
&\leq C \int_{B_{r_*}(\mathbf{y}_*)\cap D\cap\Gamma_{d+\ve}}e^{\frac{q}{h}(\tau-\frac{1}{d})}dx\\  
&\qquad+Ce^{\frac{q}{h}(\frac{1}{d+\ve}-\frac{1}{d})}
+Ce^{\frac{q}{h}(s_*-\frac{1}{d}-\delta ')}.
\end{aligned}
\]
Set 
\[
A_{q,*,h}:=\int_{B_{r_*}(\mathbf{y}_*)\cap D\cap\Gamma_{d+\ve}}e^{\frac{q}{h}(\tau-\frac{1}{d})}dx.
\]

Now we come back to (\ref{pf.L}), from Lemma \ref{lem3.1} we have for $0<h\ll 1$
\begin{align}\label{Elower}
\begin{split}
E(\bff_{d,h})&\geq C\left\{
\int_D\frac{\mu_0\mu_D}{\mu}\left|\epsilon(\bfp_{d,h})
-\frac{1}{2}(\nabla\cdot\bfp_{d,h})I_2\right|^2dx
-\|\bfw_h\|^2_{H^1(\Omega)}\right\}\\
&\quad\quad-k^2\|\bfw_{d,h}\|^2_{L^2(\Omega)}\\
&\geq C \left(\int_D\frac{\mu_0\mu_D}{\mu}\left|\epsilon(\bfp_{d,h})
-\frac{1}{2}(\nabla\cdot\bfp_{d,h})I_2\right|^2dx\right)\\
&\qquad\times\Bigg{(}1-\frac{e^{\frac{2}{h}(\frac{1}{d+\ve}-\frac{1}{d})}}
{\int_D\frac{\mu_0\mu_D}{\mu}\left|\epsilon(\bfp_{d,h})
-\frac{1}{2}(\nabla\cdot\bfp_{d,h})I_2\right|^2dx}\\
&\qquad\qquad\qquad-\frac{\|\bfw_{d,h}\|^2_{L^2(\Omega)}}
{\int_D\frac{\mu_0\mu_D}{\mu}\left|\epsilon(\bfp_{d,h})
-\frac{1}{2}(\nabla\cdot\bfp_{d,h})I_2\right|^2dx}\Bigg{)}.
\end{split}
\end{align}
In the following we estimate each term separately.

First, by Lemma \ref{est-data} we can compute
\begin{align}\label{w-est}
\begin{aligned}
&\frac{\int_{D}\frac{\mu_0\mu_D}{\mu}
\left|\epsilon(\bfp_{d,h})-\frac{1}{2}(\nabla\cdot\bfp_{d,h})I_2\right|^2dx}
{\|\bfw_{d,h}\|^2_{L^2(\Omega)}}\\
&\qquad\geq C \frac{A_{2,*,h}(c-Ch^{-2})}
{e^{\frac{2}{h}(\frac{1}{d+\ve}-\frac{1}{d})}
+h^{-4}\left(
\int_{D\cap\Gamma_{d+2\ve}}e^{\frac{q}{h}(\tau-\frac{1}{d})}dx\right)^{2/q}}\\
&\qquad\geq C \frac{A_{2,*,h}(ch^{-4}-Ch^2)}
{(A_{q,*,h})^{2/q}
+e^{\frac{2}{h}(\frac{1}{d+\ve}-\frac{1}{d})}
+e^{\frac{2}{h}(s_*-\frac{1}{d}-\delta')}},
\end{aligned}
\end{align}
for $q_0<q\leq 2$.

Now we need to compute $A_{q,*,h}$ carefully.
For $\mathbf{y}_*\in\ell_{1/s_*}\cap\partial D$, we consider the following change of coordinates as in \cite{NUW2011}.
First, let $\mathcal{T}$ be the composition of the following two rigid motions: i) translate $\mathbf{y}_*$ to the origin, and ii) rotate so that the unit inward normal of $\mathcal{T}(\Gamma_{1/s_*})$ at the origin is the vector $(0,1)^T$. Then set $\mathbf{z} = (z_1(\bfx),z_2(\bfx))^T = \mathcal{T}(\bfx)$  and $\boldsymbol{\xi}=(\xi_1(\mathbf{z}),\xi_2(\mathbf{z}))^T=\Xi(\mathbf{z})$, where
\[
\Xi(\mathbf{z})=
\left(\begin{array}{c}
z_1\\
\tau(\mathcal{T}^{-1}\mathbf{z})-s_*
\end{array}\right).
\]
Then $\Xi\circ\mathcal{T}$ gives a $\mathcal{C}^2$ diffeomorphism in a neighborhood $U_{\mathbf{y}_*}$ of $\mathbf{y}_*$. Geometrically, under the transformation $\Xi\circ\mathcal{T}$ the  point $\mathbf{y}_*$ becomes the origin of the new frame $(\xi_1,\xi_2)^T$, $\xi_1$-axis coincides with the curve $\ell_{1/s_*}$, and the positive direction of $\xi_2$-axis coincides with the unit inward normal of $\mathcal{T}(\Gamma_{1/s_*})$ at $\mathbf{y}_*$. 

We do the above change of coordinates, then we have $\Xi\circ\mathcal{T}(\mathbf{y}_*)=0$ and
\begin{align}\label{Aqh}
\begin{aligned}
&Ce^{\frac{q}{h}(s_*-\frac{1}{d})}
\left(\int_{\Xi\circ\mathcal{T}(B_{r_*}(\mathbf{y}_*)\cap D\cap\Gamma_{d+\ve})}e^{\frac{q}{h}\xi_2}d\boldsymbol{\xi}\right)\\
&\leq A_{q,*,h}
\leq Ce^{\frac{q}{h}(s_*-\frac{1}{d})}\
\left(\int_{\Xi\circ\mathcal{T}(B_{r_*}(\mathbf{y}_*)\cap D\cap\Gamma_{d+\ve})}e^{\frac{q}{h}\xi_2}d\boldsymbol{\xi}\right).
\end{aligned}
\end{align}
Since $\partial D$ is continuous, $\Xi\circ\mathcal{T}(\partial D)$ is also continuous and is able to be parametrized by a continuous function near $\boldsymbol{\xi}=\bf{0}$ under a suitable rotation.
So, we consider a rotation $\widetilde{\mathcal{T}}$ with $\widetilde{\mathcal{T}}(\boldsymbol{\xi})=\tilde{\boldsymbol{\xi}}=(\tilde{\xi}_1,\tilde{\xi}_2)^T$ such that $\widetilde{\mathcal{T}}(\Xi\circ\mathcal{T}(\partial D))$ can be parametrized by $f_*(\tilde{\xi}_1)$ near $\tilde{\boldsymbol{\xi}}=0$ with $f_*(\bf{0})=0$.

Actually, we can choose $\widetilde{\mathcal{T}}$ such that 
\[
\xi_2=(\sin\theta)\tilde{\xi}_1+(\cos\theta)\tilde{\xi}_2 \quad\mbox{with}\quad |\theta|<\frac{\pi}{2},
\]
because $\Xi\circ\mathcal{T}(D)\subset\{\xi_2\leq 0\}$ and $D$ is open.
Let $a=\sin\theta$ and $b=\cos\theta$, then $b>0$. 
Without loss of generality, we assume $\widetilde{\mathcal{T}}(\Xi\circ\mathcal{T}(\partial D))$ can be parametrized by $f_*(\tilde{\xi}_1)$ in 
$\tilde{\xi}_1<diam(\widetilde{\mathcal{T}}(\Xi\circ\mathcal{T}(B_{r_*}(\mathbf{y}_*)\cap D\cap\Gamma_{d+\ve})))$.
Set $\widetilde{U}=\widetilde{\mathcal{T}}(\Xi\circ\mathcal{T}(B_{r_*}(\mathbf{y}_*)\cap D\cap\Gamma_{d+\ve}))$.

Here we note that $f_*$ is continuous in $\widetilde{U}$ since we assume $\partial D$ has continuous boundary,
and $\widetilde{U}\subset\{a\tilde{\xi_1}+b\tilde{\xi_2}\leq 0\}$ since 
$\Xi\circ\mathcal{T}(D)\subset\{\xi_2\leq 0\}$.

Now it's easy to see that there exist positive constants $\delta_1,\delta_2,\delta_1',\delta_2'$ independent of $h$ with $\delta_2 < \delta_2'$ such that 
\[
\begin{split}
&\int_{-\delta_1'}^{\delta_1}\int_{-\delta_2}^{f_*(\tilde{\xi}_1)}
e^{\frac{q}{h}(a\tilde{\xi}_1+b\tilde{\xi}_2)}d\tilde{\boldsymbol{\xi}}\\
&\leq\int_{\Xi\circ\mathcal{T}(B_{r_*}(\mathbf{y}_*)\cap D\cap\Gamma_{d+\ve})}
e^{\frac{q}{h}\xi_2}d\boldsymbol{\xi}
=\int_{\widetilde{U}}e^{\frac{q}{h}(a\tilde{\xi_1}+b\tilde{\xi_2})}d\tilde{\boldsymbol{\xi}}\\
&\leq\int_{-\delta_1'}^{\delta_1}\int_{-\delta_2'}^{f_*(\tilde{\xi}_1)}
e^{\frac{q}{h}(a\tilde{\xi}_1+b\tilde{\xi}_2)}d\tilde{\boldsymbol{\xi}}\\
&\quad\quad=\int_{-\delta_1'}^{\delta_1}\int_{-\delta_2}^{f_*(\tilde{\xi}_1)}
e^{\frac{q}{h}(a\tilde{\xi}_1+b\tilde{\xi}_2)}d\tilde{\boldsymbol{\xi}}
+\int_{-\delta_1'}^{\delta_1}\int_{-\delta_2'}^{-\delta_2}
e^{\frac{q}{h}(a\tilde{\xi}_1+b\tilde{\xi}_2)}d\tilde{\boldsymbol{\xi}}.
\end{split}
\]

Since $\widetilde{U}\subset\{a\tilde{\xi_1}+b\tilde{\xi_2}\leq 0\}$,
\[
\delta_2\leq f_*(\tilde{\xi_1}) \leq -\frac{a}{b}\tilde{\xi_1}
\qquad\mbox{in}\quad \widetilde{U}.
\]
Therefore, we can compute directly and obtain that
\begin{align}\label{Bq}
\begin{aligned}
&\int_{-\delta_1'}^{\delta_1}\int_{-\delta_2}^{f_*(\tilde{\xi}_1)}
e^{\frac{q}{h}(a\tilde{\xi}_1+b\tilde{\xi}_2)}d\tilde{\boldsymbol{\xi}}\\
&\leq\int_{\Xi\circ\mathcal{T}(B_{r_*}(\mathbf{y}_*)\cap D\cap\Gamma_{d+\ve})}
e^{\frac{q}{h}\xi_2}d\boldsymbol{\xi}\\
&\quad\quad=\int_{-\delta_1'}^{\delta_1}\int_{-\delta_2}^{f_*(\tilde{\xi}_1)}
e^{\frac{q}{h}(a\tilde{\xi}_1+b\tilde{\xi}_2)}d\tilde{\boldsymbol{\xi}}
+(\delta_1'+\delta_1)(\delta_2'-\delta_2).
\end{aligned}
\end{align}

In order to make the computation clear, we set
\begin{align}\label{bq-def}
B_{q,*,h}=\int_{-\delta_1'}^{\delta_1}\int_{-\delta_2}^{f_*(\tilde{\xi}_1)}
e^{\frac{q}{h}(a\tilde{\xi}_1+b\tilde{\xi}_2)}d\tilde{\boldsymbol{\xi}}.
\end{align}
By combining (\ref{w-est}), (\ref{Aqh}) and (\ref{Bq}), we have
\begin{align}\label{hard1}
\begin{split}
&\frac{\int_D \frac{\mu_0\mu_D}{\mu}\left|\epsilon(\bfp_{d,h})
-\frac{1}{2}(\nabla\cdot\bfp_{d,h})I_2\right|^2dx}
{\|\bfw_{d,h}\|^2_{L^2(\Omega)}}\\
&\quad \geq C \frac{e^{\frac{2}{h}(s_*-\frac{1}{d})}B_{2,*,h}(c-Ch^2)}
{e^{\frac{2}{h}(s_*-\frac{1}{d})}(B_{q,*,h}+C)^{2/q}
+e^{\frac{2}{h}(\frac{1}{d+\ve}-\frac{1}{d})}
+e^{\frac{2}{h}(s_*-\frac{1}{d}-\delta')}}\\
&\quad =C \frac{B_{2,*,h}(c-Ch^2)}
{(B_{q,*,h})^{2/q}+C
+e^{\frac{2}{h}(\frac{1}{d+\ve}-s_*)}
+e^{\frac{2}{h}(-\delta')}}.
\end{split}
\end{align}

Now we compute $B_{q,*,h}$ more carefully. We note that since $f_*$ is continuous near $\tilde{\boldsymbol{\xi}}=0$, for all $0<\delta< \min(\delta_2,s_*\!-\!\frac{1}{d+\ve},\delta')$, there exists $0<\delta_1''<\min(\delta_1,\delta_1')$ such that
\[
|f_*(\tilde{\xi_1})|=|f_*(\tilde{\xi_1})-f_*(0)|<\delta,
\quad\forall \tilde{\xi_1}\in (-\delta_1'',\delta_1'').
\]
Therefore, 
\[
\begin{split}
B_{q,*,h}&\geq\int_{-\delta_1''}^{\delta_1''}e^{\frac{q}{h}a\tilde{\xi_1}}
\left(\int_{-\delta_2}^{f_*(\tilde{\xi_1})}
e^{\frac{q}{h}b\tilde{\xi_2}}d\tilde{\xi_2}\right)d\tilde{\xi_1}\\
&=\int_{-\delta_1''}^{\delta_1''}e^{\frac{q}{h}a\tilde{\xi_1}}\frac{h}{qb}
\left(e^{\frac{q}{h}bf_*(\tilde{\xi_1})}
-e^{-\frac{q}{h}\delta_2}\right)d\tilde{\xi_1}\\
&\geq \int_{-\delta_1''}^{\delta_1''}e^{\frac{q}{h}a\tilde{\xi_1}}\frac{h}{qb}
\left(e^{-\frac{q}{h}\delta}-e^{-\frac{q}{h}\delta_2}\right)d\tilde{\xi_1}\\
&\geq\frac{h}{qb}e^{-\frac{q}{h}b\delta}(1-e^{-\frac{q}{h}(\delta_2-\delta)})
\int_0^{\delta_1''}e^{\frac{q}{h}a\tilde{\xi_1}}d\tilde{\xi_1}. 
\end{split}
\]
Then for $0<h\ll 1$, we obtain the following estimate
\begin{align}\label{Bq-est}
B_{q,*,h}\geq C \frac{h}{qb}e^{-\frac{q}{h}b\delta},
\end{align}
for all $0<\delta<\min(\delta_2,s_*\!-\!\frac{1}{d+\ve},\delta')$ and for some $C$ independent of $h$.
Moreover, we observe that for $0<h\ll 1$
\[
\frac{e^{\frac{2}{h}(\frac{1}{d+\ve}-s_*)}}{B_{2,*,h}}
\leq Ch^{-1} e^{\frac{2}{h}(\frac{1}{d+\ve}-s_*+\delta)}
\]
and
\[
\frac{e^{-\frac{2}{h}\delta'}}{B_{2,*,h}}
\leq Ch^{-1} e^{\frac{1}{h}(-\delta'+\delta)}.
\]
Then (\ref{hard1}) becomes the following estimate
\[
\frac{\int_D \frac{\mu_0\mu_D}{\mu}\left|\epsilon(\bfp_{d,h})
-\frac{1}{2}(\nabla\cdot\bfp_{d,h})I_2\right|^2dx}
{\|\bfw_{d,h}\|^2_{L^2(\Omega)}}
\geq C\frac{B_{2,*,h}}{(B_{q,*,h})^{2/q}},
\]
for $q_0<q\leq 2$ and $0<h\ll 1$.

Actually, we can directly compute and use the H\"{o}lder inequality to obtain that for $q_0<q\leq 2$,
\[
\begin{split}
(B_{q,*,h})^{2/q}&=\left[
\int_{-\delta_1'}^{\delta_1}e^{\frac{q}{h}a\tilde{\xi_1}}\frac{h}{qb}
\left(e^{\frac{q}{h}bf_*(\tilde{\xi_1})}-e^{-\frac{q}{h}b\delta_2}
\right)d\tilde{\xi_1}\right]^{2/q}\\
&=\left(\frac{h}{qb}\right)^{2/q}\left[
\int_{-\delta_1'}^{\delta_1}e^{\frac{q}{h}a\tilde{\xi_1}}
e^{\frac{q}{h}bf_*(\tilde{\xi_1})}
\left(1-e^{-\frac{q}{h}b(\delta_2+f_*(\tilde{\xi_1}))}\right)
d\tilde{\xi_1}\right]^{2/q}\\
&\leq C\left(\frac{h}{qb}\right)^{2/q}
\int_{-\delta_1'}^{\delta_1}e^{\frac{2}{h}a\tilde{\xi_1}}
e^{\frac{2}{h}bf_*(\tilde{\xi_1})}
\left(1-e^{-\frac{q}{h}b(\delta_2+f_*(\tilde{\xi_1}))}\right)^{2/q}
d\tilde{\xi_1}.
\end{split}
\]
Since $\delta_2+f_*(\tilde{\xi_1})\geq 0$ for $\tilde{\xi_1}\in [-\delta_1',\delta_1]$, we have
\[
0<e^{-\frac{q}{h}(\delta_2+f_*(\tilde{\xi_1}))}\leq 1
\]
and therefore for $q\leq 2$
\[
\left(1-e^{-\frac{q}{h}(\delta_2+f_*(\tilde{\xi_1}))}\right)^{2/q}
\leq 1-e^{-\frac{2}{h}(\delta_2+f_*(\tilde{\xi_1}))}. 
\]
Hence we obtain for $q_0<q\leq 2$
\[
\begin{split}
(B_{q,*,h})^{2/q}&\leq C\left(\frac{h}{qb}\right)^{2/q}
\int_{-\delta_1'}^{\delta_1} e^{\frac{2}{h}a\tilde{\xi_1}} 
e^{\frac{2}{h}bf_*(\tilde{\xi_1})}
\left(1-e^{-\frac{2}{h}(\delta_2+f_*(\tilde{\xi_1}))}\right)
d\tilde{\xi_1}\\
&=C\left(\frac{h}{qb}\right)^{2/q}\left(\frac{2b}{h}\right)
\int_{-\delta_1'}^{\delta_1}\int_{-\delta_2}^{f_*(\tilde{\xi_1})}
e^{\frac{2}{h}(a\tilde{\xi_1}+b\tilde{\xi_2})}d\boldsymbol{\tilde{\xi}}\\
&=C\left(\frac{h}{qb}\right)^{2/q}\left(\frac{2b}{h}\right)B_{2,*,h},
\end{split}
\]
and then the most difficult part of the proof of Theorem \ref{thm4.1} can be concluded that for $0<h\ll 1$ and $q_0<q<2$
\begin{align}\label{hard2}
\frac{\int_D \frac{\mu_0\mu_D}{\mu}\left|\epsilon(\bfp_{d,h})
-\frac{1}{2}(\nabla\cdot\bfp_{d,h})I_2\right|^2dx}
{\|\bfw_{d,h}\|^2_{L^2(\Omega)}}
\geq C h^{1-\frac{2}{q}},
\end{align}
for some constant $C$ independent of $h$.

Back to (\ref{Elower}), by (\ref{hard2}) we have for $0<h\ll 1$ and $q_0<q<2$
\begin{align}\label{Elower1}
\begin{split}
E(\bff_{d,h})&\geq C \left(\int_D\frac{\mu_0\mu_D}{\mu}\left|\epsilon(\bfp_{d,h})
-\frac{1}{2}(\nabla\cdot\bfp_{d,h})I_2\right|^2dx\right)\\
&\quad\times\Bigg{(}1-\frac{e^{\frac{2}{h}(\frac{1}{d+\ve}-\frac{1}{d})}}
{\int_D\frac{\mu_0\mu_D}{\mu}\left|\epsilon(\bfp_{d,h})
-\frac{1}{2}(\nabla\cdot\bfp_{d,h})I_2\right|^2dx}
-h^{(\frac{2}{q}-1)}\Bigg{)}.
\end{split}
\end{align}

By direct computation we have
\[
\frac{e^{\frac{2}{h}(\frac{1}{d+\ve}-\frac{1}{d})}}
{\int_D\frac{\mu_0\mu_D}{\mu}\left|\epsilon(\bfp_{d,h})
-\frac{1}{2}(\nabla\cdot\bfp_{d,h})I_2\right|^2dx}
\leq Ce^{\frac{2}{h}(\frac{1}{d+\ve}-s_*)},
\]
therefore 
\begin{align}\label{weh}
\frac{e^{\frac{2}{h}(\frac{1}{d+\ve}-\frac{1}{d})}}
{\int_D\frac{\mu_0\mu_D}{\mu}\left|\epsilon(\bfp_{d,h})
-\frac{1}{2}(\nabla\cdot\bfp_{d,h})I_2\right|^2dx}=o(1).
\end{align}

Hence by using Lemma \ref{est-data} and by computing directly from (\ref{Elower1}) and (\ref{weh}), we have for $0<h\ll 1$ 
\begin{align}\label{conB}
|E(\bff_{d,h})| 
\geq C h^{-4}A_{2,*,h}\geq Ch^{-4}e^{\frac{2}{h}(s_*-\frac{1}{d})}B_{2,*,h}.
\end{align}
Therefore by (\ref{Bq-est}), 
for all $0<\delta <\min(\delta_2,s_*\!-\!\frac{1}{d+\ve},\delta')$ and for $0<h\ll 1$ we have
\begin{align}\label{efdhla}
|E(\bff_{d,h})|\geq Ch^{-3}e^{\frac{2}{h}(s_*-\frac{1}{d}-\delta)},
\end{align}
for some constant $C$ independent of $h$.
Choose $\delta$ such that $\delta < s_*\!-\!\frac{1}{d}$, then the proof of (B) is complete.

For case (ii) of (\ref{JC}), instead of using (\ref{pf.L}), we shall consider the negative of (\ref{pf.U}):
\[
\begin{aligned}
-E(\bff_{d,h})&\geq \int_D -(\lambda_D+\mu_D)|\nabla\cdot\bfu_{0,d,h}|^2dx\\
&\quad-2\int_D\mu_D
\left|\epsilon(\bfu_{0,d,h})-\frac{1}{2}(\nabla\cdot\bfu_{0,d,h})I_2\right|^2dx
-\int_{\Omega}k^2|\bfw_{d,h}|^2dx.
\end{aligned}
\]
And a similar argument will also give (\ref{efdhla}).

\item [(B$'$)]
As in (B), we will only prove case (i) of (\ref{JC}) by (\ref{pf.L}), and case (ii) of (\ref{JC}) can be treated similarly by using the negative of (\ref{pf.U}).
Suppose that $\bar{D}\cap\Gamma_d\neq\emptyset$ and $D$ has $C^{0,\alpha}$ boundary. 
Since $\bar{D}\cap\Gamma_d\neq\emptyset$, $s_*\geq \frac{1}{d}>\frac{1}{d+\ve}$ and $K=\partial D\cap\ell_{1/s_*}\neq\emptyset$.

In fact, we have proved in (B) that $s_*\!-\!\frac{1}{d+\ve}>0$ and continuity of $\partial D$ ensure (\ref{hard2}) holds.
So (\ref{hard2}) also holds under this assumption of (B$'$) and therefore (\ref{conB}) also holds.

However, since $D$ has $C^{0,\alpha}$ boundary, we have the better estimate than (\ref{Bq-est}).
Without loss of generality, we assume $f_*(\tilde{\xi_1})$ is $C^{0,\alpha}$ for $\tilde{\xi_1}\in [-\delta_1,\delta_1']$. 
Then there exists a positive constant $L$ such that for $\tilde{\xi_1}\in [-\delta-1,\delta_1']$
\[
|f_*(\tilde{\xi_1})|=|f_*(\tilde{\xi_1})-f_*(0)|\leq L|\tilde{\xi_1}|^{\alpha}.
\] 

Therefore we can compute directly as follows: 
\[
\begin{split}
B_{2,*,h}&=\int_{-\delta_1'}^{\delta_1}\int_{-\delta_2}^{f_*(\tilde{\xi}_1)}
e^{\frac{2}{h}(a\tilde{\xi}_1+b\tilde{\xi}_2)}d\tilde{\boldsymbol{\xi}}\\
&\geq\int_{-\delta_1'}^{0}e^{\frac{2a}{h}\tilde{\xi}_1}
\left(\int_{-\delta_2}^{f_*(\tilde{\xi}_1)}e^{\frac{2b}{h}\tilde{\xi}_2}d\tilde{\xi}_2\right)d\tilde{\xi}_1\\
&=\int_{-\delta_1'}^{0}e^{\frac{2a}{h}\tilde{\xi}_1}\frac{h}{2b}
\left(e^{\frac{2b}{h}f_*(\tilde{\xi}_1)}-e^{\frac{-2b\delta_2}{h}}\right)d\tilde{\xi}_1\\
&\geq\frac{h}{2b}
\left(\int_{-\delta_1'}^{0}
e^{\frac{2}{h}(a\tilde{\xi}_1-bL|\tilde{\xi_1}|^{\alpha})}
-e^{\frac{2a}{h}\tilde{\xi}_1}e^{\frac{-2b\delta_2}{h}}d\tilde{\xi}_1\right)\\
&=\frac{h}{2b}
\left(\int_{0}^{\delta_1'}
e^{-\frac{2}{h}(a\tilde{\xi}_1+bL\tilde{\xi_1}^{\alpha})}
-e^{-\frac{2a}{h}\tilde{\xi}_1}e^{\frac{-2b\delta_2}{h}}d\tilde{\xi}_1\right).
\end{split}
\]
Without loss of generality, we assume that $0<\delta_1<1$. 
Since $0<\alpha\leq 1$, we have
\[
\begin{split}
\int_0^{\delta_1'}e^{\frac{-2}{h}
(a\tilde{\xi_1}+bL\tilde{\xi_1}^{\alpha})}d\tilde{\xi_1}
&\geq \int_0^{\delta_1'}e^{\frac{-2}{h}
(a+bL)\tilde{\xi_1}^{\alpha}}d\tilde{\xi_1}\\
&=h^{\frac{1}{\alpha}}\int_0^{\frac{\delta_1'}{h^{1/\alpha}}}
e^{-2(a+bL)\tilde{\xi_1}^{\alpha}}d\tilde{\xi_1}.
\end{split}
\]
Then by computing directly, we have
\[
B_{2,*,h}\geq \frac{h}{2b}
\left\{h^{\frac{1}{\alpha}}
\int_0^{\frac{\delta_1'}{h^{1/\alpha}}}e^{-2(a+bL)\tilde{\xi_1}^{\alpha}}d\tilde{\xi}_1
-he^{\frac{-2\delta_2}{h}}\int_0^{\frac{\delta_1'}{h}}
e^{-2a\tilde{\xi}_1}d\tilde{\xi}_1\right\}.
\]
Since
\[
\int_0^{\frac{\delta_1'}{h^{1/\alpha}}}e^{-2(a+bL)\tilde{\xi_1}^{\alpha}}d\tilde{\xi}_1
\rightarrow\int_0^{\infty}e^{-2(a+bL)\tilde{\xi}_1^{\alpha}}d\tilde{\xi}_1<\infty
\quad\mbox{as}\quad h\rightarrow 0^+
\]
and
\[
\int_0^{\frac{\delta_1}{h}}e^{-2a\tilde{\xi}_1}d\tilde{\xi}_1
\rightarrow\int_0^{\infty}e^{-2a\tilde{\xi}_1}d\tilde{\xi}_1<\infty
\quad\mbox{as}\quad h\rightarrow 0^+,
\]
we have
\[
B_{2,*,h}\geq C h^{1+\frac{1}{\alpha}} \quad\mbox{for}\quad 0<h\ll 1.
\]
Then by (\ref{conB})
\[
E(\bff_{d,h})\geq C e^{\frac{2}{h}(s_*-\frac{1}{d})}
h^{-4}h^{1+\frac{1}{\alpha}}
=Ce^{\frac{2}{h}(s_*-\frac{1}{d})}h^{-3+\frac{1}{\alpha}}
\]
for each $0<h\ll 1$.
If $\alpha>\frac{1}{3}$, then even when $s_*\!=\!\frac{1}{d}$, 
$|E(\bff_{d,h})|$ tends to infinity as $h$ tends to zero.
\end{itemize}
\end{proof}

\section*{Acknowledgements}
I would like to deeply thank Professor Jenn-Nan Wang for the suggestions of this problem and for many helpful discussions.
I would also like to thank Professor Gen Nakamura for informing me about the work of Sini and Yoshida. 
Thank Liren Lin for many discussions on my English writing.
This work is partly supported by the National Science Council
of the Republic of China.

\bibliographystyle{plain}
\bibliography{ref}

\begin{thebibliography}{10}

\bibitem{C1980}
A.P. Calder{\'o}n.
\newblock On an inverse boundary value problem.
\newblock In {\em Seminar on {N}umerical {A}nalysis and its {A}pplications to
  {C}ontinuum {P}hysics ({R}io de {J}aneiro, 1980)}, pages 65--73. Soc. Brasil.
  Mat., Rio de Janeiro, 1980.

\bibitem{E1998book}
Lawrence~C. Evans.
\newblock {\em Partial differential equations}, volume~19 of {\em Graduate
  Studies in Mathematics}.
\newblock American Mathematical Society, Providence, RI, 1998.

\bibitem{IINSU2007}
T.~Ide, H.~Isozaki, S.~Nakata, S.~Siltanen, and G.~Uhlmann.
\newblock Probing for electrical inclusions with complex spherical waves.
\newblock {\em Comm. Pure Appl. Math.}, 60(10):1415--1442, 2007.

\bibitem{I1998}
M.~Ikehata.
\newblock Reconstruction of the shape of the inclusion by boundary
  measurements.
\newblock {\em Comm. Partial Differential Equations}, 23(7-8):1459--1474, 1998.

\bibitem{I1999-2}
M.~Ikehata.
\newblock Enclosing a polygonal cavity in a two-dimensional bounded domain from
  {C}auchy data.
\newblock {\em Inverse Problems}, 15(5):1231--1241, 1999.

\bibitem{I1999}
M.~Ikehata.
\newblock How to draw a picture of an unknown inclusion from boundary
  measurements. {T}wo mathematical inversion algorithms.
\newblock {\em J. Inverse Ill-Posed Probl.}, 7(3):255--271, 1999.

\bibitem{I2001}
M.~Ikehata.
\newblock The enclosure method and its applications.
\newblock In {\em Analytic extension formulas and their applications
  ({F}ukuoka, 1999/{K}yoto, 2000)}, volume~9 of {\em Int. Soc. Anal. Appl.
  Comput.}, pages 87--103. Kluwer Acad. Publ., Dordrecht, 2001.

\bibitem{I2005-2}
M.~Ikehata.
\newblock An inverse transmission scattering problem and the enclosure method.
\newblock {\em Computing}, 75(2-3):133--156, 2005.

\bibitem{II2008}
M.~Ikehata and H.~Itou.
\newblock An inverse problem for a linear crack in an anisotropic elastic body
  and the enclosure method.
\newblock {\em Inverse Problems}, 24(2):025005, 21, 2008.

\bibitem{LN2003}
Y.~Li and L.~Nirenberg.
\newblock Estimates for elliptic systems from composite material.
\newblock {\em Comm. Pure Appl. Math.}, 56(7):892--925, 2003.
\newblock Dedicated to the memory of J{\"u}rgen K. Moser.

\bibitem{M2000book}
W.~McLean.
\newblock {\em Strongly elliptic systems and boundary integral equations}.
\newblock Cambridge University Press, Cambridge, 2000.

\bibitem{M1963}
N.~G. Meyers.
\newblock An {$L^{p}$}-estimate for the gradient of solutions of second order
  elliptic divergence equations.
\newblock {\em Ann. Scuola Norm. Sup. Pisa (3)}, 17:189--206, 1963.

\bibitem{NUW2011}
Sei Nagayasu, Gunther Uhlmann, and Jenn-Nan Wang.
\newblock Reconstruction of penetrable obstacles in acoustic scattering.
\newblock {\em SIAM Journal on Mathematical Analysis}, 43(1):189--211, 2011.

\bibitem{NY2007}
G.~Nakamura and K.~Yoshida.
\newblock Identification of a non-convex obstacle for acoustical scattering.
\newblock {\em J. Inverse Ill-Posed Probl.}, 15(6):611--624, 2007.

\bibitem{SW2006}
M.~Salo and J.N. Wang.
\newblock Complex spherical waves and inverse problems in unbounded domains.
\newblock {\em Inverse Problems}, 22(6):2299--2309, 2006.

\bibitem{SYpreprint}
M.~Sini and K.~Yoshida.
\newblock On the reconstruction of interfaces using cgo solutions for the
  acoustic case.
\newblock {\em Preprint}.

\bibitem{UW2007}
G.~Uhlmann and J.N. Wang.
\newblock Complex spherical waves for the elasticity system and probing of
  inclusions.
\newblock {\em SIAM J. Math. Anal.}, 38(6):1967--1980 (electronic), 2007.

\bibitem{UW2008}
G.~Uhlmann and J.N. Wang.
\newblock Reconstructing discontinuities using complex geometrical optics
  solutions.
\newblock {\em SIAM J. Appl. Math.}, 68(4):1026--1044, 2008.

\bibitem{UWW2009}
G.~Uhlmann, J.N. Wang, and C.T. Wu.
\newblock Reconstruction of inclusions in an elastic body.
\newblock {\em J. Math. Pures Appl. (9)}, 91(6):569--582, 2009.

\bibitem{V1967book}
I.~N. Vekua.
\newblock {\em New methods for solving elliptic equations}.
\newblock Translated from the Russian by D. E. Brown. Translation edited by A.
  B. Tayler. North-Holland Publishing Co., Amsterdam, 1967.
\newblock North-Holland Series in Applied Mathematics and Mechanics, Vol. 1.

\end{thebibliography}

\end{document}